\documentclass[11pt]{amsart}
\usepackage[T1]{fontenc}
\usepackage{lmodern,microtype,amsmath,amssymb,mathtools,mathrsfs}
\usepackage[margin=1.05in]{geometry}
\usepackage{booktabs,array,enumitem}
\usepackage{float}
\usepackage{longtable}
\usepackage{xcolor}
\usepackage[colorlinks=true,linkcolor=blue!55!black,citecolor=blue!55!black,urlcolor=blue!55!black]{hyperref}
\usepackage[capitalise,nameinlink,noabbrev]{cleveref}
\hypersetup{pdftitle={Diamond maximal operators: classification, moduli and oscillation envelopes},
pdfauthor={Mohammadreza Mohajer},pdfsubject={Harmonic analysis on Euclidean spaces and local fields}}
\newcolumntype{P}[1]{>{\raggedright\arraybackslash}p{#1}}
\newtheorem{theorem}{Theorem}[section]
\newtheorem{proposition}[theorem]{Proposition}
\newtheorem{corollary}[theorem]{Corollary}
\newtheorem{lemma}[theorem]{Lemma}
\theoremstyle{definition}
\newtheorem{definition}[theorem]{Definition}
\newtheorem{example}[theorem]{Example}
\theoremstyle{remark}
\newtheorem{remark}[theorem]{Remark}
\numberwithin{equation}{section}
\newcommand{\R}{\mathbb R}
\newcommand{\Z}{\mathbb Z}

\newcommand{\C}{\mathbb C}
\newcommand{\Ocal}{\mathcal O}
\newcommand{\one}{\mathbf 1}
\newcommand{\avint}{\mathop{\mathchoice{\rlap{\kern.2em\raise.1ex\hbox{$-$}}\int}{\rlap{\kern.1em\raise.1ex\hbox{$-$}}\int}{\int}{\int}}}

\newcommand{\M}{\mathcal M}
\newcommand{\D}{\mathcal D}
\newcommand{\E}{\mathbb E}
\newcommand{\BMO}{\mathrm{BMO}}
\newcommand{\BLO}{\mathrm{BLO}}
\newcommand{\loc}{\mathrm{loc}}
\newcommand{\res}{\mathrm{res}}
\newcommand{\ang}{\mathrm{ang}}
\newcommand{\ac}{\operatorname{ac}}
\newcommand{\osc}{\operatorname{osc}}

\newcommand{\ran}{\operatorname{Im}}
\newcommand{\rank}{\operatorname{rank}}
\newcommand{\tr}{\operatorname{tr}}
\newcommand{\spanof}{\operatorname{span}}
\newcommand{\conv}{\operatorname{conv}}
\newcommand{\essinf}{\operatorname*{ess\,inf}}

\newcommand{\Gr}{\operatorname{Gr}}

\newcommand{\norm}[1]{\lVert #1\rVert}
\newcommand{\pair}[2]{\langle #1,#2\rangle}
\newcommand{\doi}[1]{\href{https://doi.org/#1}{\nolinkurl{#1}}}
\setlist[enumerate]{label=\textup{(\roman*)},leftmargin=*}
\title[Diamond maximal operators]{Diamond maximal operators:\\ classification, moduli, and oscillation envelopes}
\author{Mohammadreza Mohajer}
\address{Department of Mathematics, Physics and Geology, Cape Breton University,
1250 Grand Lake Road, Sydney, NS B1P 6L2, Canada}
\email{Mohammadreza\_Mohajer@cbu.ca}
\email{mohammadreza.mohajer96@gmail.com}
\date{}
\begin{document}
\begin{abstract}
We study maximal averages generated by bounded cancellative
kernels taking values in finite-dimensional real normed spaces,
on Euclidean balls and on balls in a non-Archimedean local field
with a fixed uniformizer.
For angular kernels, pointwise domination of the full maximal
operators is equivalent to factorization of their angular
averaging maps, with the same optimal constant. Nonzero scalar
angular comparison classes are therefore projective kernel classes.
For angular kernels constant on a fixed finite measurable
partition of the direction space, symmetric convex bodies
classify exact norm measurements, Grassmannians parametrize
nonzero classes under mutual domination, and positive semidefinite
Gram matrices classify Hilbert-valued measurements.
For general kernels, the closed convex hull of the averaging
densities obtained from all scales and the dual unit ball
is a complete invariant of the full maximal operator.
A dilation example exhibits identifications of residue kernels
beyond scaling and unit-coordinate changes.
For nonzero radial residue kernels on bounded inputs, mutual
domination is classified by the unit-circle zeros of filter
polynomials, counted with multiplicity, with optimal constants
given by the $\ell^1$ norms of convolution quotients.
We identify the envelope over all normalized bounded cancellative
kernels with the best local $L^1$ approximation by constants
and prove quantitative convergence of residue envelopes as the
resolution increases. Finally, the weights of the finite Fourier transform classify
the one-level energy and the associated Carleson functional.
This functional is equivalent to the $\BMO_2$ seminorm precisely
when every nontrivial additive character of the residue field
is detected, and the weights determine the optimal comparison
constants.
\end{abstract}
\maketitle
\section{Introduction}
Maz'ya and Shaposhnikova introduced the diamond maximal operator
\begin{equation}\label{eq:mazya}
 M^\diamond f(x)=\sup_{r>0}
 \left\|\frac{1}{|B(x,r)|}\int_{B(x,r)}
 \frac{y-x}{|y-x|}f(y)\,dy\right\|_{\R^d}
\end{equation}
in their study of sharp pointwise interpolation inequalities
\cite{MazyaShaposhnikova2002}. The vector sign kernel has mean zero,
so $M^\diamond f\leq M^\sharp f$, where $M^\sharp$ is the centered
Fefferman--Stein sharp maximal function
\cite{FeffermanStein1972,CoifmanFefferman1974}. Kova{\v c} constructed
a translation-invariant Banach space on which $M^\diamond$ is bounded
and $M^\sharp$ is unbounded \cite{Kovac2026BanachDistinguishes},
answering a question of Maz'ya \cite{Mazya2018}. This separation
motivates our study of the information measured by cancellative
kernels and their dilation families.

We retain the translated and dilated ball averages in
\eqref{eq:mazya} and replace the vector sign kernel by an arbitrary
bounded mean-zero kernel with values in a finite-dimensional real
normed space $E$. We study this construction on $\R^d$ and extend it
to a non-Archimedean local field $K$, with ring of integers
$\Ocal$ and a fixed uniformizer $t$.
\Cref{def:diamond} gives the resulting operators. The output norm
combines scalar measurements, while cancellation removes constants
from the input. Kova{\v c}'s
Euclidean boundedness argument already includes compactly supported
bounded vector kernels with components of bounded variation
\cite[Section~2]{Kovac2026BanachDistinguishes}. Here we study
pointwise comparison and equality for bounded kernels without a
regularity assumption, together with the corresponding local-field
models.

We ask when kernels generate equal or mutually dominating maximal
operators, and how these classifications change under geometric
changes of coordinates. By default, domination is required
pointwise for every locally integrable input.
For general kernels, support-function duality gives a complete
invariant: the closed convex hull of the averaging densities
obtained from every scale and the dual unit ball. Inclusion of
these convex sets is equivalent to
pointwise domination, and equality of the sets is equivalent to
equality of the operators (\cref{thm:convex-invariant}). For finite
families of measurements, the optimal comparison constant is given
by a linear program (\cref{prop:lp}).

For angular kernels, classification reduces to a linear
factorization problem. The direction space $\Omega$ is $S^{d-1}$
in Euclidean space and $\Ocal^\times$ over $K$; the corresponding
directions of $z\ne0$ are respectively, $z/|z|$ and
$\ac_t(z)=t^{-v(z)}z$, with $v(t)=1$.
A bounded mean-zero function on this space gives an angular kernel
and its lift to the reference ball (\cref{def:angular-kernel}).
Write $L_\theta h=\int_\Omega\theta h\,d\nu$ for its angular
averaging map, with $\nu$ the normalized direction measure.
\Cref{thm:angular-rigidity} proves that domination of the full
angular maximal operators is equivalent to a factorization
$L_\theta=A L_\psi$, where $A:\ran(L_\psi)\to E$, with the
optimal constant equal to $\|A\|$.
In particular, nonzero scalar kernels are comparable precisely
when they are proportional (\cref{cor:scalar}). Their comparison
classes are therefore projective kernel classes; quotienting by
orthogonal transformations or multiplication of directions by units
gives the geometric moduli in \eqref{eq:angular-moduli}.

For angular kernels constant on the cells of a fixed finite
measurable partition of the direction space, norm duality classifies exact equality by
symmetric convex bodies and mutual domination by their linear spans
(\cref{thm:finite-moduli}) which the latter classes form Grassmannians.
For inner-product output norms, Gram matrices classify exact
measurements and their ranks give the minimal output dimensions
(\cref{thm:gram}). In the Euclidean family
$\theta_A(\omega)=A\omega$, positive scaling and orthogonal changes
of input coordinates reduce the classification to a simplex of
normalized eigenvalues (\cref{cor:euclidean-simplex}).
Over $K$, the finite quotients
$A_M=\Ocal^\times/(1+t^M\Ocal)$ give projective moduli and moduli
of normalized Gram matrices under the finite unit-group action.
These moduli embed compatibly as the angular resolution increases
(\cref{prop:angular-levels}). Scalar projective stabilizers record
symmetries that can distinguish fields with the same residue
cardinality (\cref{prop:stabilizers,ex:arithmetic}).

Residue kernels instead factor through the full ball quotient
$R_N=\Ocal/t^N\Ocal$ (\cref{def:residue-kernel}). Their values
record radial as well as angular information, separated by the
equivariant decomposition in \cref{prop:decomposition}.
Finite truncations connect locally constant angular kernels to
these residue models (\cref{prop:angular-truncation}).
Projective residue spaces classify scalar kernels up to scaling
and unit-coordinate changes, but a dilation can preserve the full
maximal operator while changing that kernel class
(\cref{ex:dilation-redundancy}). The distinction between a kernel
and its dilation family is therefore essential.

For radial residue kernels, full comparison on bounded inputs has
an explicit polynomial classification
(\cref{thm:radial-classification}). Radial averages become finite
filters of nested ball-average differences. Mutual domination is
classified by the zeros of the filter polynomials on the unit
circle, counted with multiplicity and exact equality is classified
by the filter up to sign and shift. The optimal comparison
constant is the $\ell^1$ norm of the convolution quotient.
We use division in the Wiener algebra of absolutely convergent
Fourier series \cite{Wiener1932}
to characterize comparison of finite filters. Realizing finitely
supported sequences of ball-average differences by bounded radial
inputs transfers this criterion to the corresponding maximal
operators with the same optimal comparison constants.
\Cref{ex:radial-domain} shows why this classification uses
$L^\infty(K)$ rather than all locally integrable inputs.

We then ask how families of measurements recover local
oscillation. Duality with constants identifies the envelope over
all normalized bounded cancellative kernels with the best local
$L^1$ approximation by a constant
(\cref{thm:median,cor:full-envelope}). Wilson's intrinsic square
function provides a related kernel supremum over tests satisfying
a H\"older condition \cite{Wilson2007}.
As the residue resolution increases, the envelopes converge
monotonically to the full envelope; on each ball, the error is
bounded by the conditional-averaging error
(\cref{thm:residue-limit}). Angular envelopes measure instead the
distance of the radial profile from constants and do not uniformly
recover full oscillation (\cref{prop:radial-blindness}).
Even residue tests of any fixed finite resolution fail to control
sharp oscillation: sums of $L$ independent digit functions give
exact maximal values and a discrepancy of order $\sqrt{L/N}$
for $L\geq N$ (\cref{thm:resolution-obstruction}).

Quadratic accumulation across scales gives a different recovery
mechanism. The local-field BMO and Carleson background is developed
in \cite{DalyKurtz1993BMOlocalfields,Kim2009}. Orthogonal
ball-average decompositions connect the construction with Haar and
$p$-adic wavelet analysis \cite{Kozyrev2002}.
For translates of a fixed level-one residue kernel, the finite
Fourier weights classify the one-level energy
(\cref{prop:fourier-energy}) and the full Carleson functional
(\cref{thm:energy}). Recovery of $\BMO_2$ is equivalent to
detection of every nontrivial additive character of the residue
field, with optimal constants given by the smallest and largest
weights. The resulting spectral moduli retain squared Fourier
amplitudes, and mutual comparison retains their support
(\cref{cor:energy-moduli}).
\Cref{app:guide} collects the distinct equivalence relations and
their complete invariants.

\section{Diamond structures and elementary estimates}\label{sec:setup}
All input functions are real-valued. Output spaces are finite-dimensional
real normed spaces unless stated otherwise. Integrals of vector-valued
functions are defined componentwise. Kernels and functions are identified
almost everywhere.

\subsection{The two geometric models}
Let $X=\R^d$, with Lebesgue measure, or let $X=K$ be a
non-Archimedean local field with norm $|\cdot|_K$.
Thus $K$ is a nondiscrete locally compact field whose topology is
induced by $|\cdot|_K$, and
\[
 |x+y|_K\leq\max\{|x|_K,|y|_K\}
 \qquad(x,y\in K).
\]
Its ring of integers and the unique maximal ideal of that ring are
\[
 \Ocal=\{x\in K:|x|_K\leq1\},
 \qquad
 \mathfrak m=\{x\in K:|x|_K<1\},
\]
respectively. The ring $\Ocal$ is compact, the ideal $\mathfrak m$
is principal, and the residue field $\Ocal/\mathfrak m$ is finite of prime power order $q$
\cite[Chapters~I--II]{SerreLocalFields}.
Fix a uniformizer $t$, meaning a generator of $\mathfrak m$. We normalize the absolute value by
$|t|_K=q^{-1}$ and additive Haar measure $\mu$ by $\mu(\Ocal)=1$;
see \cite{Taibleson1975} for the harmonic-analysis background.
When working over $K$, we write $|\cdot|$ for this normalized norm.
We use the notation
\[
 \begin{array}{c|c|c|c}
 X& B_0&\text{scale parameter }s&\delta_s z\\ \hline
 \R^d&B(0,1)&r>0&rz\\
 K&\Ocal&k\in\Z&t^kz
 \end{array}
\]
and set $B_s(x)=x+\delta_s B_0$, $b_s=\mu(\delta_s B_0)$.
Thus $b_r=|B(0,1)|r^d$ in Euclidean space and $b_k=q^{-k}$ over $K$.
Normalized measure on $B_0$ is denoted by $\nu_0=\frac{\mu|_{B_0}}{\mu(B_0)}$.
Whenever a ball kernel is regarded as a function on $X$, it is
extended by zero outside $B_0$.
Every ball containing $x$ in $K$ is one of the balls $B_k(x)$. Over $K$, each ball $B_k(x)$ is the disjoint union of $q$ balls
of scale $k+1$, called its children. Its descendants of depth $N$
are the $q^N$ balls of scale $k+N$ contained in it. Each such
descendant has measure $q^{-N}\mu(B_k(x))$. Thus increasing the
scale index $k$ corresponds to passing to smaller balls.

The classical diamond operator suggests a general construction based
on a cancellative pattern on a reference ball. In Euclidean space,
put
\[
 \theta_0(z)=\frac{z}{|z|}\quad(z\ne0),
 \qquad \theta_0(0)=0.
\]
The kernel $\theta_0$ is bounded and has mean zero on $B_0$ by
symmetry. The change of variables $y=x+rz$ rewrites
\eqref{eq:mazya} as
\[
 M^\diamond f(x)
 =\sup_{r>0}
   \left\|
     \int_{B_0}\theta_0(z)f(x+rz)\,d\nu_0(z)
   \right\|_{\R^d},
\]
where $\R^d$ carries its Euclidean norm.

This representation separates the geometry of the balls from the
pattern used to measure the input. We retain translation, dilation,
and normalization by ball measure, and allow the pattern to be any
bounded kernel with mean zero. Boundedness makes each average
well-defined for locally integrable inputs, while cancellation makes
the averages insensitive to the addition of constants. Allowing
values in a finite-dimensional real normed space accommodates both
scalar measurements and collections of measurements combined through
an output norm. Over a local field, the same construction uses
$t^kz$ in place of $rz$.

\begin{definition}\label{def:diamond}
Let $(E,\|\cdot\|_E)$ be a finite-dimensional real normed space.
A diamond kernel with values in $E$ is a function
$\theta\in L^\infty(B_0;E)$ satisfying
\[
 \int_{B_0}\theta(z)\,d\nu_0(z)=0.
\]
For a real-valued input $f\in L^1_{\loc}(X)$, its associated
average at center $x$ and scale $s$ is
\begin{align}
 T_{\theta,s}f(x)
 &=\int_{B_0}\theta(z)f(x+\delta_s z)\,d\nu_0(z)\notag\\
 &=\frac{1}{b_s}\int_{B_s(x)}
   \theta\bigl(\delta_s^{-1}(y-x)\bigr)f(y)\,d\mu(y).
 \label{eq:average}
\end{align}
The associated diamond maximal operator is
\begin{equation}\label{eq:diamond}
 \D_\theta f(x)=\sup_s\|T_{\theta,s}f(x)\|_E,
\end{equation}
where the supremum runs over the scales specified above.
For a uniformly bounded family of such kernels, the associated
envelope is obtained by taking the supremum over the kernels
as well as the scales.
\end{definition}

For $E=\R$, the construction gives scalar diamond operators.
Taking $E=\R^d$ with its Euclidean norm and $\theta=\theta_0$
recovers \eqref{eq:mazya}. More generally, if
$E=\R^r$ and $\theta=(\theta_1,\ldots,\theta_r)$, the coordinates
of $T_{\theta,s}f(x)$ are the scalar averages generated by the
kernels $\theta_j$, and the norm on $E$ combines these measurements.
The orientation $y-x$ is fixed throughout; reversing it replaces
$\theta(z)$ by $\theta(-z)$.

For a ball $B$, write $f_B=\mu(B)^{-1}\int_Bf\,d\mu$ and
\[
 \osc_B(f)=\frac{1}{\mu(B)}\int_B|f-f_B|\,d\mu,
 \qquad
 \delta_B(f)=\inf_{c\in\R}\frac{1}{\mu(B)}\int_B|f-c|\,d\mu.
\]
Set
\[
 \M f(x)=\sup_s\frac{1}{b_s}\int_{B_s(x)}|f|\,d\mu,
 \qquad \M^\sharp f(x)=\sup_s\osc_{B_s(x)}(f).
\]

\begin{proposition}\label{prop:basic}
For every $f\in L^1_{\loc}(X)$,
\begin{equation}\label{eq:basic}
 \D_\theta f\leq\norm{\theta}_\infty\sup_s\delta_{B_s(\cdot)}(f)
 \leq\norm{\theta}_\infty\M^\sharp f,
 \qquad \D_\theta f\leq\norm{\theta}_\infty\M f.
\end{equation}
Moreover $\D_\theta(f+c)=\D_\theta f$ and
$\D_\theta(af)=|a|\D_\theta f$.
\end{proposition}
\begin{proof}
Cancellation permits subtraction of any constant $c$ in
\eqref{eq:average}. Taking the norm, minimizing over $c$, and taking the
supremum over $s$ proves the first inequalities. Applying the triangle
inequality without subtracting a constant proves the last one. The two
invariances follow directly.
\end{proof}

The operators above are measurable. Over $K$ the scale set is countable and
each fixed-scale average is continuous in $x$: on a compact set of centers,
translation continuity in $L^1$ applies to a compactly truncated input,
while the kernel is bounded. In Euclidean space the same argument gives
continuity in $x$; continuity in $r>0$ follows from strong continuity of
translations and dilations in $L^1$ on compact parameter ranges. It suffices
to take rational radii. General envelopes used below are measurable either
as suprema of continuous functions or by their explicit oscillation formulas.
Values $+\infty$ are allowed. Products of nonnegative extended values
are interpreted with $0\cdot(+\infty)=0$.

\subsection{BMO, BLO, and the role of order}
We use the seminorms
\[
 \norm{f}_{\BMO_1}=\sup_B\osc_B(f),\qquad
 \norm{g}_{\BLO}=\sup_B\left(g_B-\essinf_B g\right),
\]
where all balls are included and $g$ is locally integrable. The first is the
usual mean-oscillation seminorm \cite{JohnNirenberg1961}.
For $f\in L^1_{\loc}(X)$ with $\norm{f}_{\BMO_1}<\infty$,
\cref{prop:basic} gives
\begin{equation}\label{eq:bmo-blo}
 \norm{\D_\theta f}_\infty\leq\norm{\theta}_\infty\norm{f}_{\BMO_1},
 \qquad
 \norm{\D_\theta f}_{\BLO}
 \leq\norm{\theta}_\infty\norm{f}_{\BMO_1}.
\end{equation}
The first estimate makes $\D_\theta f$ bounded and locally integrable.
Since $\D_\theta f$ is nonnegative, its $\BLO$ seminorm is at most
$\norm{\D_\theta f}_\infty$, which proves the second estimate.

\begin{remark}\label{rem:blo}
The $\BLO$ seminorm is not monotone under pointwise order. If
$0<\mu(A)<\mu(B)$ for a measurable subset $A$ of a ball $B$, then
$0\leq\one_A\leq1$,
whereas $\norm{1}_{\BLO}=0$ and
$\norm{\one_A}_{\BLO}\geq\mu(A)/\mu(B)>0$.
Consequently a pointwise comparison of maximal functions cannot by itself
be promoted to a comparison of their $\BLO$ seminorms.
\end{remark}

Over $K$, the standard disjoint-ball argument gives
\[
 \mu\{\M f>\lambda\}\leq\lambda^{-1}\norm{f}_1,
 \qquad \norm{\M f}_p\leq\frac{p}{p-1}\norm{f}_p\quad(1<p<\infty).
\]
For completeness, for $f\in L^1(K)$ the averages tend to zero on balls
whose measures tend to infinity. Maximal balls with average greater than
$\lambda$ therefore exist, are disjoint, and give the stronger estimate
$\lambda\mu\{\M f>\lambda\}\leq\int_{\{\M f>\lambda\}}|f|$.
Integration against $p\lambda^{p-2}d\lambda$, followed by H\"older's
inequality, gives the $L^p$ estimate, first with finite scale and spatial
truncations and then by monotone convergence and $L^p$ approximation.
\Cref{prop:basic} transfers both bounds to $\D_\theta$, with factor
$\norm{\theta}_\infty$. These estimates provide analytic boundedness. The classification below concerns the more restrictive condition of
pointwise comparison for every input.

Unless another input class is specified, pointwise domination means
an inequality
\[
 \D_\theta f(x)\leq C\D_\psi f(x)
\]
for every $f\in L^1_{\loc}(X)$ and every $x\in X$, with a finite
constant $C\geq0$ independent of $f$ and $x$. The operators are
mutually comparable when domination holds in both directions,
possibly with different constants. Exact equality means equality
of their nonnegative outputs for every input and center.
In almost-everywhere comparisons, the exceptional null set may
depend on the input.

\section{The invariant of a general maximal operator}\label{sec:convex}
For a finite-dimensional real normed space $E$, let $E^*$ denote
its real linear dual, equipped with the dual norm
\[
 \|\ell\|_{E^*}
 =\sup_{\|e\|_E\leq1}|\ell(e)|.
\]
Norm duality gives
\[
 \|e\|_E
 =\sup_{\|\ell\|_{E^*}\leq1}|\ell(e)|
 =\sup_{\|\ell\|_{E^*}\leq1}\ell(e).
\]
Thus vector-valued measurements can be studied through their scalar
projections. For an $E$-valued kernel, $\|\theta\|_\infty$ denotes
the essential supremum of $\|\theta(z)\|_E$. We call the kernel
normalized when $\|\theta\|_\infty\leq1$.

At a fixed center, each scalar projection of a diamond average is
a linear functional of the input. The maximal operator is the
supremum of these functionals over the scales and the dual unit
ball. This supremum is unchanged by taking convex combinations
of the averaging densities. It is unchanged by taking a closure
in a topology for which pairing with every admissible input is
continuous. These observations lead to the following convex
invariant.

Let $\mathscr Z=L^\infty_c(X;\R)$ denote essentially bounded functions
with compact essential support, and let $\mathscr F=L^1_{\loc}(X;\R)$.
These spaces form a separating dual pair under
$\pair{a}{f}=\int_Xaf\,d\mu$. Give $\mathscr Z$ the weak topology
$\sigma(\mathscr Z,\mathscr F)$; this is the topology of convergence of
these pairings for every $f\in\mathscr F$.
For a kernel $\theta$, define
\begin{align*}
 \mathscr A_\theta
 &=\left\{b_s^{-1}\ell(\theta(\delta_s^{-1}\,\cdot))
          \one_{\delta_sB_0}:s,\ \ell\in E^*,\ \norm{\ell}\leq1\right\},\\
 \mathscr C_\theta
 &=\overline{\conv(\mathscr A_\theta)}^{\,\sigma(\mathscr Z,\mathscr F)}.
\end{align*}
The generating set is symmetric and contains zero. All closures in
this section are taken inside $\mathscr Z$.

\begin{theorem}[Complete convex invariant]\label{thm:convex-invariant}
For diamond kernels $\theta,\psi$, possibly with different output spaces,
and $C>0$, the following statements are equivalent:
\begin{enumerate}
\item $\D_\theta f(x)\leq C\D_\psi f(x)$ for every
$f\in L^1_{\loc}(X)$ and every $x\in X$;
\item $\mathscr C_\theta\subseteq C\mathscr C_\psi$.
\end{enumerate}
In particular, the two maximal operators are equal on all inputs if and
only if their closed convex invariants are equal.
\end{theorem}
\begin{proof}
Duality of the finite-dimensional output norm gives
\[
 \D_\theta f(0)=\sup_{a\in\mathscr A_\theta}\pair{a}{f}
              =\sup_{a\in\mathscr C_\theta}\pair{a}{f}.
\]
The last equality holds because the pairing is continuous in the specified
topology. Inclusion of the convex sets gives domination at zero, and applying
that assertion to $f(x+\cdot)$ gives domination at $x$.

Conversely, suppose $a_0\in\mathscr C_\theta$ lies outside
$C\mathscr C_\psi$. Separation of a point from a closed convex set in the
locally convex space $\mathscr Z$ gives a continuous linear functional
strictly separating them. Every continuous linear functional for this weak
topology is pairing with an element of $\mathscr F$. For that element $f$,
\[
 \D_\theta f(0)\geq\pair{a_0}{f}
 >\sup_{a\in C\mathscr C_\psi}\pair{a}{f}=C\D_\psi f(0),
\]
a contradiction. Applying the equivalence in both directions proves the
equality statement.
\end{proof}

This is the support-function form of the Hahn--Banach separation theorem.
Its importance here is that the invariant contains \emph{all dilates}, not
just the original kernel.

\begin{example}[Dilation redundancy]\label{ex:dilation-redundancy}
Over $K$, extend a diamond kernel $\theta$ by zero outside
$\Ocal$ and define
\[
 \psi(z)=q\theta(t^{-1}z)\one_{t\Ocal}(z).
\]
The substitution $z=tw$ gives
\[
 T_{\psi,k}f=T_{\theta,k+1}f,
 \qquad
 \D_\psi=\D_\theta.
\]
For example, take
\[
 \theta=\one_{t\Ocal}-q^{-1}\one_\Ocal,
 \qquad
 \psi=q\one_{t^2\Ocal}-\one_{t\Ocal}.
\]
On $\Ocal\setminus t\Ocal$, the first kernel equals
$-q^{-1}$ and the second vanishes. Multiplication by a unit
preserves this set, so no nonzero scalar multiple of a
unit-coordinate change of $\theta$ equals $\psi$.
Thus equality of the full maximal operators admits
identifications beyond scaling and unit-coordinate changes.
\end{example}

The convex comparison criterion becomes explicit when finitely many
scalar averaging densities are retained and these densities are
constant on a common finite partition. In this setting, each
measurement depends only on the averages of the input over the
partition cells.

Fix a center and consider two finite collections of compactly
supported scalar averaging densities
$\alpha_1,\ldots,\alpha_u$ and $\beta_1,\ldots,\beta_v$.
Suppose that these densities are constant on a common measurable
partition $P_1,\ldots,P_m$ of the union of their supports, with
$0<\mu(P_\ell)<\infty$. Define coefficient vectors
$a_i,b_j\in\R^m$ by
\[
 (a_i)_\ell=\int_{P_\ell}\alpha_i\,d\mu,
 \qquad
 (b_j)_\ell=\int_{P_\ell}\beta_j\,d\mu.
\]
For $f\in L^1_{\loc}(X)$, write
\[
 f_{P_\ell}
 =\frac{1}{\mu(P_\ell)}\int_{P_\ell}f\,d\mu,
 \qquad
 F=(f_{P_1},\ldots,f_{P_m}).
\]
Constancy on the partition cells gives
\[
 \int_X\alpha_i f\,d\mu=a_i\cdot F,
 \qquad
 \int_X\beta_j f\,d\mu=b_j\cdot F.
\]
Every vector $F\in\R^m$ is realized by taking $f=F_\ell$ on
$P_\ell$ and $f=0$ outside the union of the cells. Consequently,
comparison of the two finite families is equivalent to an inequality
between the corresponding maxima of linear functionals on $\R^m$.

A representation $a_i=\sum_j\lambda_jb_j$ certifies the estimate
\[
 |a_i\cdot F|
 \leq
 \left(\sum_j|\lambda_j|\right)
 \max_j|b_j\cdot F|.
\]
The next proposition shows that minimizing the total absolute
coefficient size gives the optimal comparison constant.

\begin{proposition}[Finite comparison certificates]\label{prop:lp}
With the notation above, the least constant $C\geq0$ satisfying
\[
 \max_{1\leq i\leq u}|a_i\cdot F|
 \leq C\max_{1\leq j\leq v}|b_j\cdot F|
 \qquad(F\in\R^m)
\]
is
\[
 C_*=
 \max_{1\leq i\leq u}
 \inf\left\{
     \sum_{j=1}^v|\lambda_j|:
     a_i=\sum_{j=1}^v\lambda_jb_j,\ 
     \lambda_j\in\R
 \right\},
\]
where the infimum of an empty feasible set is $+\infty$.
In particular, a finite comparison constant exists precisely when
\[
 \spanof\{a_1,\ldots,a_u\}
 \subseteq
 \spanof\{b_1,\ldots,b_v\}.
\]
\end{proposition}

\begin{proof}
Let
\[
 K_b=\operatorname{conv}\{\pm b_1,\ldots,\pm b_v\}.
\]
This compact, convex, symmetric set has the representation
\[
 K_b=
 \left\{
   \sum_{j=1}^v\lambda_jb_j:
   \sum_{j=1}^v|\lambda_j|\leq1
 \right\},
\]
and its support function is
\[
 \sup_{b\in K_b}b\cdot F
 =\max_{1\leq j\leq v}|b_j\cdot F|.
\]

For a fixed $C\geq0$, the stated comparison holds if
$a_i\in CK_b$ for every $i$, by the support-function identity
and the symmetry of $K_b$. Conversely, if some $a_i$ lies
outside $CK_b$, finite-dimensional separation gives a vector
$F\in\R^m$ such that
\[
 a_i\cdot F
 >
 \sup_{b\in CK_b}b\cdot F
 =
 C\max_{1\leq j\leq v}|b_j\cdot F|,
\]
contradicting the comparison. Thus domination with constant $C$
is equivalent to $a_i\in CK_b$ for every $i$.

The coefficient representation of $K_b$ now gives the formula
for $C_*$. Whenever a representation is feasible, the infimum
is attained: a minimizing sequence of coefficient vectors has
bounded $\ell^1$ norm and therefore has a convergent subsequence
whose limit remains feasible. Feasibility for every $i$ is
equivalent to the asserted inclusion of linear spans.

Finally, each minimization is a linear program. Writing
$\lambda_j=\lambda_j^+-\lambda_j^-$ gives
\[
 \begin{aligned}
 \text{minimize}\quad&
   \sum_{j=1}^v(\lambda_j^++\lambda_j^-),\\
 \text{subject to}\quad&
   \sum_{j=1}^v(\lambda_j^+-\lambda_j^-)b_j=a_i,\\
 & \lambda_j^+,\lambda_j^-\geq0
   \qquad(1\leq j\leq v).
 \end{aligned}
\]
Taking the maximum of these optimal values over $i$ yields $C_*$.
\end{proof}

\section{Angular rigidity and operator moduli}\label{sec:angular}
\subsection{Direction spaces and radial profiles}
In Euclidean space let $(\Omega,\nu)=(S^{d-1},\sigma)$, where $\sigma$ is surface measure divided by the total surface area of the
sphere; for $d=1$ it gives mass $1/2$ to each of
$-1,1$. Thus integration against
$\sigma$ gives the average over directions. Over $K$, let $\Omega=\Ocal^\times$ be the space of directions with normalized additive Haar
measure $\nu=(1-q^{-1})^{-1}\mu|_{\Ocal^\times}$. This is invariant under
multiplication by units. Here
$\Ocal^\times=\{u\in\Ocal:|u|=1\}$
is the group of units. For $z\ne0$, its normalized valuation
$v(z)\in\Z$ gives $|z|=q^{-v(z)}$. Define its angular component by
\[
 \ac_t(z)=t^{-v(z)}z\in\Ocal^\times.
\] Then we have 
 $z=t^{v(z)}\ac_t(z)$.
\begin{definition}[Angular kernel and angular lift]
\label{def:angular-kernel} 
An \emph{angular kernel} is a function $\theta\in L^\infty(\Omega;E)$ with $\int_\Omega\theta\,d\nu=0$.
Its \emph{angular lift} to $B_0$ is $\widetilde\theta(z):=\theta(z/|z|)$ in Euclidean space and
$\widetilde\theta(z):=\theta(\ac_t(z))$ over $K$. We set the lift equal to zero at the origin, whose value does
not affect the integrals defining $T_{\theta,s}f(x)$.
Write $\D^\ang_\theta$ for the associated diamond operator generated by this
lift.
\end{definition}
The lift assigns the same kernel value to points with the same
direction. Polar integration gives
\[
 \int_{B_0}\widetilde\theta\,d\nu_0
 =\int_\Omega\theta\,d\nu=0,
\]
so $\widetilde\theta$ is a diamond kernel in the sense of
\cref{def:diamond}. For an angular kernel $\theta$, we abbreviate
$T_{\widetilde\theta,s}$ by $T_{\theta,s}$. Define the bounded linear map $L_\theta:L^1(\Omega,\nu;\R)\to E$ by
\[
 L_\theta h=\int_\Omega\theta(\omega)h(\omega)\,d\nu(\omega).
\]
At a fixed center and scale, we first average the input in the
radial variable while holding the direction fixed. The resulting
function on $\Omega$ is the radial profile $P_{x,s}f$.
The angular measurement is then obtained by applying $L_\theta$
to that profile. This gives
\begin{equation}\label{eq:angular-profile}
 T_{\theta,s}f(x)=L_\theta(P_{x,s}f),
\end{equation}
where the profiles are
\begin{align}
 P_{x,r}f(\omega)&=d\int_0^1 f(x+r\rho\omega)\rho^{d-1}\,d\rho,
                    &&X=\R^d,\label{eq:euclidean-profile}\\
 P_{x,k}f(u)&=(1-q^{-1})\sum_{j=0}^\infty q^{-j}f(x+t^{k+j}u),
                    &&X=K.\label{eq:local-profile}
\end{align}
The series converges absolutely for almost every $u$, and the profiles
belong to $L^1(\Omega)$ by Fubini's theorem.

\begin{theorem}[Angular rigidity]\label{thm:angular-rigidity}
Let $E$ and $F$ be finite-dimensional real normed spaces, and let $\theta\in L^\infty(\Omega;E)$ and
$\psi\in L^\infty(\Omega;F)$ be mean-zero angular kernels. For $C\geq0$
the following are equivalent:
\begin{enumerate}
\item $\D^\ang_\theta f(x)\leq C\D^\ang_\psi f(x)$ for every
$f\in L^1_{\loc}(X)$ and every $x$;
\item $\norm{L_\theta h}_E\leq C\norm{L_\psi h}_F$ for every
$h\in L^1(\Omega)$;
\item there is a linear map $A:\ran(L_\psi)\to E$, of norm at most $C$,
such that $L_\theta=A L_\psi$.
\end{enumerate}
The map in \textup{(iii)} is unique on $\ran (L_\psi)$, and its norm is the
least comparison constant. Existence of a finite comparison constant is
equivalent to $\ker L_\psi\subseteq\ker L_\theta$.
The same classification results if \textup{(i)} is required only almost
everywhere for each input.
\end{theorem}
\begin{proof}
For $h\in L^1(\Omega)$ define
$f_h(z)=h(z/|z|)\one_{B_0}(z)$ over $\R^d$ and
$f_h(z)=h(\ac_t(z))\one_{\Ocal}(z)$ over $K$.
At the origin, polar integration gives
\[
 T_{\theta,r}f_h(0)=\min(1,r^{-d})L_\theta h,
 \qquad
 T_{\theta,k}f_h(0)=\min(1,q^k)L_\theta h,
\]
respectively. Consequently $\D^\ang_\theta f_h(0)=\norm{L_\theta h}$.
This proves \textup{(i)}$\Rightarrow$\textup{(ii)}. Formula
\eqref{eq:angular-profile} proves the reverse implication.

Condition \textup{(ii)} permits the well-defined assignment
$A(L_\psi h)=L_\theta h$, and its bound is exactly $C$.
Conversely such a factorization gives \textup{(ii)}. If only the kernel
inclusion is known, the same assignment is a linear map on the
finite-dimensional normed space $\ran L_\psi$ and is therefore bounded.
Every element of $\ran L_\psi$ has the form $L_\psi h$, so the
assignment determines $A$ uniquely. Taking the supremum of
$\norm{L_\theta h}/\norm{L_\psi h}$ over nonzero images gives
the assertion about the best constant. If $\ran L_\psi=\{0\}$,
the factorization map is zero and has norm zero.

For the almost-everywhere version, use continuous angular $h$ in Euclidean
space and a nonnegative continuous radial cutoff supported in an annulus.
The resulting $f$ is continuous and compactly supported, and
$T_{\theta,r}f(0)=a_rL_\theta h$, with $a_r\geq0$ and
$0<\sup_ra_r<\infty$. Every such $f$ is uniformly continuous, and
\[
 |\D^\ang_\theta f(x)-\D^\ang_\theta f(x')|
 \leq\norm{\theta}_\infty
       \sup_z|f(x+z)-f(x'+z)|.
\]
Thus the two maximal functions are continuous, and an almost-everywhere
inequality holds at the origin. Density of continuous functions in
$L^1(S^{d-1})$ gives \textup{(ii)}. Over $K$, use locally constant angular
$h$ and the cutoff $\one_{\Ocal^\times}$. These produce continuous compactly
supported inputs and the same positive radial factor. Locally constant
functions are dense in $L^1(\Ocal^\times)$.
\end{proof}
The seminorm $h\mapsto\|L_\theta h\|_E$ induces a norm on the
finite-dimensional quotient
\[
 L^1(\Omega,\nu;\R)/\ker L_\theta.
\]
With this induced norm, the quotient is naturally isometric to
$\ran(L_\theta)$ with its induced output norm.
Thus angular comparison concerns the norms and kernels of the
angular averaging maps.
\begin{corollary}[Scalar classification]\label{cor:scalar}
Two nonzero real scalar angular maximal operators admit a pointwise
comparison in either direction if and only if their angular kernels are
nonzero scalar multiples. For $\theta=c\psi$, the optimal comparison
constant is $|c|$.
\end{corollary}
\begin{proof}
A nonzero scalar functional has one-dimensional range. The factorization
in \cref{thm:angular-rigidity} is multiplication by a scalar. Equality of
the corresponding integrals for all $h\in L^1$ identifies the kernels
almost everywhere.
\end{proof}

Let $G=O(d)$ be the orthogonal group of $\R^d$ in the Euclidean model and $G=\Ocal^\times$ in the local-field
model $K$, acting on directions in $\Omega$ the usual way. Define
$(g\cdot\theta)(\omega)=\theta(g^{-1}\omega)$. A change of variables gives
\begin{equation}\label{eq:covariance}
 \D^\ang_{g\cdot\theta}f(x)
 =\D^\ang_\theta(f\circ g)(g^{-1}x).
\end{equation}
Hence the set
\begin{equation}\label{eq:angular-moduli}
 \mathbb P\big(L^\infty_0(\Omega;\R)\big)/G
\end{equation}
classifies nonzero scalar angular operators up to mutual pointwise domination
and these geometric changes of coordinates. Here $L^\infty_0$ denotes the
mean-zero subspace and $\mathbb P(V)$ means quotient of $V\setminus\{0\}$ by $\R^\times$. It denotes real projective space. Thus a projective point represents
a nonzero kernel up to multiplication by a nonzero real scalar.
For general measurable kernels, \eqref{eq:angular-moduli} is the set
of orbits under scaling and the direction-group action.
The finite-dimensional moduli spaces below carry the usual quotient
topology.

\subsection{Convex bodies, Grassmannians, and Gram matrices}

A finite angular model records the values of a kernel on the fibers
of a fixed quotient of the direction space. Let
$\Lambda=\{1,\ldots,m\}$, with $m\geq2$, carry the uniform probability
measure, and fix a measurable map
\[
 \pi:\Omega\to\Lambda,
 \qquad
 A_j=\pi^{-1}(\{j\}),
 \qquad
 \nu(A_j)=\frac1m.
\]
Thus the sets $A_j$ partition the original direction space $\Omega$,
and $\Lambda$ is the finite set of their labels. The map $\pi$ and
the labels are fixed throughout the comparisons below.

Let $E$ be a finite-dimensional real normed space. A list
$v=(v_1,\ldots,v_m)\in E^m$ defines an angular kernel
\[
 \theta_v(\omega)=v_{\pi(\omega)}.
\]
Conversely, every angular kernel that is constant almost everywhere
on each $A_j$ has this form, up to equality almost everywhere.
Cancellation is equivalent to
\[
 \int_\Omega\theta_v\,d\nu
 =\frac1m\sum_{j=1}^m v_j=0.
\]
For such a mean-zero list, define
\[
 L_v:\R^m\to E,
 \qquad
 L_vF=\frac1m\sum_{j=1}^m v_jF_j,
 \qquad
 E_v=\spanof\{v_1,\ldots,v_m\}.
\]
The space $E_v$, equipped with the norm inherited from $E$, is the
effective output space: every measurement lies in $E_v$, and the
vectors $L_v(me_j)=v_j$, where $e_j$ is the $j$th standard basis
vector, span it.

The connection with angular integration is given by the
cell-averaging map
\[
 Q_\pi:L^1(\Omega,\nu;\R)\to\R^m,
 \qquad
 (Q_\pi h)_j=m\int_{A_j}h\,d\nu.
\]
Direct integration gives
\begin{equation}\label{eq:finite-angular-factorization}
 L_{\theta_v}=L_vQ_\pi.
\end{equation}
Moreover, $Q_\pi$ is surjective: for any $F\in\R^m$, the function
$h_F(\omega)=F_{\pi(\omega)}$ satisfies $Q_\pi h_F=F$.
Consequently, for two mean-zero lists $v,w$ lifted through the same
map $\pi$, the inequality
\[
 \norm{L_{\theta_v}h}\leq C\norm{L_{\theta_w}h}
 \qquad\bigl(h\in L^1(\Omega,\nu;\R)\bigr)
\]
is equivalent to
\[
 \norm{L_vF}\leq C\norm{L_wF}
 \qquad(F\in\R^m),
\]
with the same optimal constant. By
\cref{thm:angular-rigidity}, these conditions are equivalent to
pointwise domination of the corresponding full angular maximal
operators, with that constant.

We now describe the finite-dimensional invariants. Put
\[
 H_m=\left\{a\in\R^m:\sum_{j=1}^m a_j=0\right\},
 \qquad
 \pair{a}{F}_m=\frac1m\sum_{j=1}^m a_jF_j.
\]
Associate to a mean-zero list $v$ the set
\[
 C_v=
 \left\{
  \bigl(\ell(v_1),\ldots,\ell(v_m)\bigr):
  \ell\in E^*,\ \norm{\ell}_{E^*}\leq1
 \right\}
 \subseteq H_m.
\]
This set collects the scalar coefficient lists obtained from the
dual unit ball. It is compact, convex, and symmetric about the
origin, meaning $C_v=-C_v$. Convex bodies below are understood
relative to their linear spans. We write $\Gr(r,H_m)$ for the
Grassmannian of $r$-dimensional real linear subspaces of $H_m$.

\begin{theorem}[Finite-dimensional moduli]\label{thm:finite-moduli}
Let $v$ and $w$ be mean-zero lists on $\Lambda$, with values in
finite-dimensional real normed spaces $E$ and $E'$, respectively.
The following statements hold.
\begin{enumerate}
\item For every $F\in\R^m$,
\[
 \norm{L_vF}_E=\sup_{a\in C_v}\pair{a}{F}_m.
\]
For $C\geq0$, the inequality
$\norm{L_vF}_E\leq C\norm{L_wF}_{E'}$ for every $F$ is
equivalent to $C_v\subseteq C C_w$. In particular, the seminorms
$F\mapsto\norm{L_vF}_E$ and $F\mapsto\norm{L_wF}_{E'}$ are
equal if and only if $C_v=C_w$.

\item Every nonempty compact convex set $C\subseteq H_m$ satisfying
$C=-C$ is realized as $C_v$, allowing the output space and its norm
to vary. The minimal effective output dimension is
$\dim\spanof C$.

\item Two nonzero seminorms $F\mapsto\norm{L_vF}_E$ and
$F\mapsto\norm{L_wF}_{E'}$ are mutually comparable if and only if
\[
 \spanof C_v=\spanof C_w.
\]
Allowing all finite-dimensional real normed output spaces, their
mutual-comparison classes are therefore parametrized by
\[
 \bigsqcup_{r=1}^{m-1}\Gr(r,H_m).
\]
\end{enumerate}
For angular kernels lifted through the fixed map $\pi$, these
classifications give exact equality and mutual pointwise domination
of the full angular maximal operators. 
\end{theorem}

\begin{proof}
Norm duality gives
\[
 \norm{L_vF}_E
 =\sup_{\norm{\ell}_{E^*}\leq1}
       \frac1m\sum_{j=1}^m\ell(v_j)F_j
 =\sup_{a\in C_v}\pair{a}{F}_m.
\]
The support functions of compact convex sets determine their
inclusions by finite-dimensional separation. Applying this fact to
$C_v$ and $C C_w$ proves \textup{(i)}, including $C=0$.

For \textup{(ii)}, put $V=\spanof C$. If $C=\{0\}$, take the
zero-dimensional output space. Otherwise, compactness, convexity,
symmetry, and the spanning property imply that $C$ is the unit
ball of the norm
\[
 \norm{a}_C=\inf\{\lambda>0:a\in\lambda C\}
 \qquad(a\in V).
\]
Give $E=V^*$ the dual norm, and define $v_j\in E$ by
$v_j(a)=a_j$. Since $V\subseteq H_m$, these coordinate
functionals satisfy $\sum_jv_j=0$. Under the canonical
identification $E^*=V$, the dual unit ball of $E$ is $C$, and
\[
 \bigl(a(v_1),\ldots,a(v_m)\bigr)
 =(a_1,\ldots,a_m)
 \qquad(a\in V).
\]
Thus $C_v=C$.

For any realization, the map
\[
 J_v:E^*\to\R^m,
 \qquad
 J_v(\ell)=\bigl(\ell(v_1),\ldots,\ell(v_m)\bigr)
\]
has image $\spanof C_v$ and kernel equal to the annihilator of
$E_v$. Hence $\dim\spanof C_v=\dim E_v$. The construction above
attains the asserted dimension.

For \textup{(iii)}, mutual comparison is equivalent by
\textup{(i)} to inclusion of each convex body in a positive dilate
of the other, which forces equality of their spans. Conversely,
two compact symmetric convex sets spanning the same space are
unit balls of equivalent norms on that space, so these inclusions
hold. Every nonzero subspace of $H_m$ occurs by \textup{(ii)},
for example by realizing its Euclidean unit ball. This gives the
Grassmannian parametrization.

Finally, \eqref{eq:finite-angular-factorization}, surjectivity of
$Q_\pi$, and \cref{thm:angular-rigidity} transfer the finite
comparisons to the full angular maximal operators.
\end{proof}

When the output norm comes from an inner product, the norm of
$L_vF$ is determined by the pairwise inner products of the vectors
$v_j$. Their \emph{Gram matrix} is
\[
 G_v=\bigl(\langle v_i,v_j\rangle_E\bigr)_{i,j=1,\ldots, m}.
\]
Its diagonal entries are the squared lengths of the kernel values,
and its off-diagonal entries record their inner products. For a
second list $w$ in a real Hilbert space $E'$, define $G_w$ in the
same way. Write $\one=(1,\ldots,1)^{\mathsf T}$. For real symmetric
matrices, $G\succeq0$ means
$F^{\mathsf T}GF\geq0$ for every $F\in\R^m$, and
$G\preceq H$ means $H-G\succeq0$.

\begin{theorem}[Hilbert-valued classification]\label{thm:gram}
Let $v$ and $w$ be mean-zero lists on $\Lambda$ with values in
finite-dimensional real Hilbert spaces.
\begin{enumerate}
\item The Gram matrix satisfies
\begin{equation}\label{eq:gram}
 \norm{L_vF}^2=m^{-2}F^{\mathsf T}G_vF,
 \qquad G_v\succeq0,
 \qquad G_v\one=0.
\end{equation}
Every real symmetric matrix $G\succeq0$ satisfying $G\one=0$
is realized by a mean-zero list, with minimal effective output
dimension $\rank G$.

\item For $C\geq0$,
\[
 \norm{L_vF}\leq C\norm{L_wF}
 \quad\text{for every }F\in\R^m
 \quad\Longleftrightarrow\quad
 G_v\preceq C^2G_w.
\]
Equality of the two seminorms is equivalent to $G_v=G_w$.
The same conditions characterize comparison and equality of the
full angular maximal operators lifted through $\pi$.
Moreover, $G_v=G_w$ if and only if there is a linear isometry
$U:E_v\to E_w$ satisfying $Uv_j=w_j$ for every $j$; this isometry
is unique and onto.

\item A finite comparison constant exists if and only if
$\ker G_w\subseteq\ker G_v$. If $G_w\ne0$ and this inclusion
holds, the least constant satisfies
\begin{equation}\label{eq:gram-constant}
 C_*^2=
 \sup_{\substack{F\in\R^m\\F^{\mathsf T}G_wF>0}}
 \frac{F^{\mathsf T}G_vF}{F^{\mathsf T}G_wF}.
\end{equation}
If both matrices vanish, $C_*=0$; if the kernel inclusion fails,
$C_*=+\infty$.

\item For a nonzero list, the normalized Gram matrix
$G_v/\tr G_v$ removes positive scaling of the kernel.
For $1\leq r\leq m-1$, the space of trace-one matrices
$G\succeq0$ with $G\one=0$ and $\rank G=r$ has dimension
\[
 \frac{r(2(m-1)-r+1)}2-1.
\]
\end{enumerate}
\end{theorem}

\begin{proof}
Expanding the squared norm proves the first identity in
\eqref{eq:gram}. A Gram matrix is positive semidefinite, and
$\sum_jv_j=0$ gives $G_v\one=0$.

Conversely, factor an admissible matrix as $G=V^{\mathsf T}V$,
where $V$ has $\rank G$ rows, and take the columns of $V$ as
the kernel values. Since
\[
 \norm{V\one}^2=\one^{\mathsf T}G\one=0,
\]
the columns sum to zero. Their span has dimension $\rank G$,
which is necessary in every realization because the rank of a
Gram matrix equals the dimension of the span of its vectors.

The comparison assertion in \textup{(ii)} follows directly from
\eqref{eq:gram}. Equality of the quadratic forms determines the
symmetric matrices by polarization. If $G_v=G_w$, the assignment
\[
 \sum_j a_jv_j\longmapsto\sum_j a_jw_j
\]
is well-defined and norm preserving: both squared norms equal
$a^{\mathsf T}G_va$. It is a linear isometry onto $E_w$ and
is uniquely determined by the spanning list. Conversely, such an
isometry preserves every Gram entry. The conclusions for angular
maximal operators follow from the fixed-quotient factorization
and \cref{thm:angular-rigidity}.

If the kernel inclusion in \textup{(iii)} fails, some
$F\in\ker G_w$ satisfies $F^{\mathsf T}G_vF>0$, so no finite
comparison constant exists. If the inclusion holds and $G_w\ne0$,
both quadratic forms depend only on the component of $F$ in
$(\ker G_w)^\perp$, where $G_w$ is positive definite. The
quotient in \eqref{eq:gram-constant} is therefore bounded, and
its supremum is exactly the least squared comparison constant.
The case of two zero matrices is immediate.

For \textup{(iv)}, scaling the list by $a>0$ multiplies its Gram
matrix by $a^2$. Put $d_0=m-1$. An admissible matrix restricts to
a positive semidefinite form on $H_m$. A rank-$r$ such matrix is
specified by its $r$-dimensional image, contributing
$r(d_0-r)$ parameters, and a positive definite symmetric form on
that image, contributing $r(r+1)/2$ parameters. Trace normalization
removes one parameter. Their sum gives the stated dimension.
\end{proof}

Thus, for the fixed angular quotient, Gram matrices classify
exact norm measurements, and their images classify mutual
domination. In the scalar case $G_v=vv^{\mathsf T}$, its normalized
rank-one matrix represents the projective class of $v$ in
$\mathbb P(H_m)$.

The geometric action on a finite model depends on the partition.
If a geometric transformation $g$ permutes the cells $A_j$ up to
null sets, let $P_g$ be its induced permutation matrix, with
$(P_gv)_j=v_{\sigma_g^{-1}(j)}$ when $gA_j=A_{\sigma_g(j)}$.
Then
\[
 G_{P_gv}=P_gG_vP_g^{\mathsf T}.
\]
This gives the action of the subgroup preserving the partition.
A general geometric transformation transports the partition
together with the kernel. For the local-field quotient $A_M$
below, multiplication by every unit permutes the fibers, and the
induced finite action is the action of $A_M$.

\begin{example}[The role of the angular partition]
\label{ex:angular-partition}
On $\Omega=S^1$ with normalized surface measure, consider
\[
 \theta(\omega)=\operatorname{sgn}(\omega_1),
 \qquad
 \psi(\omega)=\operatorname{sgn}(\omega_2).
\]
Each kernel is represented by the list $(1,-1)$ on its own
partition into two semicircles. These two lists have the same
Gram matrix
\[
 \begin{pmatrix}1&-1\\-1&1\end{pmatrix}.
\]
The quotient maps are different, however, and
\[
 L_\theta\theta=1,\qquad L_\psi\theta=0,
 \qquad
 L_\psi\psi=1,\qquad L_\theta\psi=0.
\]
By \cref{thm:angular-rigidity}, neither angular maximal operator
is dominated by a finite multiple of the other with the input
coordinates fixed.

On the common partition into the four quadrants, listed
counterclockwise starting with the first quadrant, the kernels
are represented by
\[
 v=(1,-1,-1,1)^{\mathsf T},
 \qquad
 w=(1,1,-1,-1)^{\mathsf T}.
\]
Their Gram matrices are $vv^{\mathsf T}$ and $ww^{\mathsf T}$,
with distinct one-dimensional images. The fixed-quotient
classification therefore detects their incomparability.
A quarter-turn carries $\theta$ to $\psi$, so the two kernels
are geometrically equivalent under the separate coordinate
equivalence.
\end{example}

\begin{remark}[Refining angular partitions]
\label{rem:angular-refinement}
Kernels presented on different finite partitions can be compared
on a common refinement. The refined cells need not have equal
measure. If their positive measures are $p_1,\ldots,p_n$, write
$v_j$ for the kernel values and $F_j$ for the cell averages.
Then the relevant map and cancellation condition are
\[
 L_{v,p}F=\sum_{j=1}^n p_jv_jF_j,
 \qquad
 \sum_{j=1}^n p_jv_j=0.
\]
These weighted formulas reduce to the finite linear algebra above
by setting $\widetilde v_j=np_jv_j$, since
$\sum_j\widetilde v_j=0$ and $L_{\widetilde v}=L_{v,p}$.
For Hilbert-valued lists, with
$D_p=\operatorname{diag}(p_1,\ldots,p_n)$, the corresponding
quadratic identity is
\[
 \norm{L_{v,p}F}^2
 =F^{\mathsf T}D_pG_vD_pF,
 \qquad
 G_{\widetilde v}=n^2D_pG_vD_p.
\]
Thus comparison on a refinement retains the actual cell measures.
\end{remark}
\subsection{An explicit Euclidean moduli space}
Equip $\R^d$ and $\R^r$ with their Euclidean norms. For a real matrix
$A:\R^d\to\R^r$, put $\theta_A(\omega)=A\omega$ and
$G_A=A^{\mathsf T}A$.

\begin{corollary}\label{cor:euclidean-simplex}
The maximal operators $\D^\ang_{\theta_A}$ and
$\D^\ang_{\theta_B}$ are equal if and only if $G_A=G_B$.
For $C\geq0$, the inequality
$\D^\ang_{\theta_A}f\leq C\D^\ang_{\theta_B}f$ holds for every locally
integrable input if and only if $G_A\preceq C^2G_B$. Up to positive scaling and orthogonal changes of
input coordinates, the nonzero operators are classified by
\begin{equation}\label{eq:simplex}
 \{(\lambda_1,\ldots,\lambda_d):
     \lambda_1\geq\cdots\geq\lambda_d\geq0,\
     \lambda_1+\cdots+\lambda_d=1\}.
\end{equation}
The coordinates $\lambda_i$ are the eigenvalues of
$G_A/\tr G_A$, listed in decreasing order.
If output dimension is bounded by $r$, retain only points with at most $r$
nonzero coordinates. Mutual domination, after an orthogonal coordinate
change, is classified solely by rank.
\end{corollary}
\begin{proof}
The map $h\mapsto\int_{S^{d-1}}\omega h(\omega)\,d\sigma$ is onto
$\R^d$: for $h(\omega)=d\pair{a}{\omega}$ its value is $a$, because
$\int\omega\omega^{\mathsf T}d\sigma=I/d$. Apply
\cref{thm:angular-rigidity} to the maps $a\mapsto Aa$ and $a\mapsto Ba$.
The spectral theorem then identifies orthogonal conjugacy classes by
eigenvalues. Trace normalization gives \eqref{eq:simplex}.
Two positive semidefinite forms are mutually comparable exactly when their
kernels agree, and orthogonal transformations identify subspaces of the
same dimension.
\end{proof}

The original vector sign kernel is $A=I_d$ and corresponds to
$(1/d,\ldots,1/d)$. Scalar coordinate kernels correspond to rank-one
forms. Two nonparallel scalar coordinate directions are incomparable before
rotating coordinates, as follows immediately from the corollary.

\subsection{Angular moduli over finite unit quotients for a local field}
For $M\geq1$ let
\[
 A_M=\Ocal^\times/(1+t^M\Ocal),\quad m_M=(q-1)q^{M-1},\quad
 V_M=\left\{\theta:A_M\to\R:\sum_{a\in A_M}\theta(a)=0\right\}.
\]
An angular kernel factors through $A_M$ precisely when it is
constant on every coset $u(1+t^M\Ocal)$.
Equivalently, its value depends only on the residue class of the
unit direction modulo $t^M\Ocal$.
Every coset has normalized unit measure $1/m_M$. Scalar angular
comparability classes, up to multiplication of directions by units, are
\begin{equation}\label{eq:finite-angular-moduli}
 \mathfrak M_M^\ang=\mathbb P(V_M)/A_M.
\end{equation}
For $m_M\geq2$ this is a compact quotient of
$\R\mathbb P^{m_M-2}$ by the finite group $A_M$. If $m_M=1$, $V_M=0$ and the moduli space is empty.
In particular there is no nonzero level-one angular kernel when $q=2$.

We next consider kernels with values in a finite-dimensional real
Hilbert space $E$, meaning a real inner-product space equipped
with the norm
\[
 \norm{e}_E=\sqrt{\langle e,e\rangle_E}.
\]
Fix the quotient $A_M$ and put $m=m_M$. An angular kernel
factoring through $A_M$ is represented by a list
$v=(v_a)_{a\in A_M}$ satisfying $\sum_a v_a=0$.
Its Gram matrix is
\[
 G_v=\bigl(\langle v_a,v_b\rangle_E\bigr)_{a,b\in A_M}.
\]
By Gram-matrix classification in
\cref{thm:gram}, this matrix determines the list up to a linear
isometry of its effective output space
$E_v=\spanof\{v_a:a\in A_M\}$. Conversely, every real symmetric
positive semidefinite matrix $G$ satisfying $G\one=0$ is the
Gram matrix of such a mean-zero list, allowing the output
dimension to vary.

We identify two nonzero lists $v$ and $w$ when
\[
 w_a=s\,U v_{u^{-1}a}
 \qquad(a\in A_M)
\]
for some $s>0$, some $u\in A_M$, and a linear isometry
$U:E_v\to E_w$ onto the effective output space of $w$.
These operations represent positive scaling, an isometric
change of output coordinates, and multiplication of the unit
directions by $u$, respectively.

Let $P_u$ be the permutation matrix defined by
 $(P_uF)(a)=F(u^{-1}a)$.
The corresponding transformation of Gram matrices is
 $G_w=s^2P_uG_vP_u^{\mathsf T}$.
Since a nonzero list satisfies
\[
 \tr G_v=\sum_{a\in A_M}\norm{v_a}_E^2>0,
\]
positive scaling is removed by the normalization
 $\widehat G_v=\frac{G_v}{\tr G_v}$.
Consequently, the two lists are equivalent precisely when
\[
 \widehat G_w=P_u\widehat G_vP_u^{\mathsf T}
 \qquad\text{for some }u\in A_M.
\]
Indeed, this identity gives
$G_w=s^2P_uG_vP_u^{\mathsf T}$ with
$s=(\tr G_w/\tr G_v)^{1/2}$, and equality of Gram matrices
then supplies the required output isometry.

The normalized moduli space is therefore
\[
 \mathscr G_m/A_M,
 \qquad
 \mathscr G_m=
 \left\{
  G\in\R^{A_M\times A_M}:
  G=G^{\mathsf T},\ G\succeq0,\ G\one=0,\ \tr G=1
 \right\},
\]
where $A_M$ acts by $G\mapsto P_uGP_u^{\mathsf T}$.
For output dimension at most $r$, the corresponding space is
\[
 \{G\in\mathscr G_m:\rank G\leq r\}/A_M,
\]
because $\rank G_v=\dim E_v$.
These spaces are compact: the positive semidefinite trace-one
matrices form a compact set, the rank bound is closed, and $A_M$ is finite. Finally, \cref{thm:angular-rigidity,thm:gram} identify these
same spaces with the classes of the corresponding angular
maximal operators under positive scaling and unit-coordinate changes.

Passing from level $M$ to level $M+1$ refines the angular
resolution. Let $\pi_M:A_{M+1}\to A_M$
be reduction. Each fiber of this surjective group homomorphism
has $q$ elements. Its pullback matrix is the linear map
\[
 P_M:\R^{A_M}\to\R^{A_{M+1}},
 \qquad
 (P_MF)(b)=F(\pi_M(b)).
\]
Thus $P_M$ repeats each coarse-level value on its $q$ finer
labels. Its columns have disjoint supports and squared length
$q$, so $P_M^{\mathsf T}P_M=qI$.
In particular, $P_M$ is injective and maps $V_M$ into $V_{M+1}$.

For a normalized Gram matrix $G\in\mathscr G_{m_M}$, the
unnormalized Gram matrix of the pulled-back list is
$P_MGP_M^{\mathsf T}$, whose trace is $q$. Its normalized
representative is therefore
$\Phi_M(G)=q^{-1}P_MGP_M^{\mathsf T}$.
For a group $\Gamma$ acting on a set $Y$, write
\[
 \operatorname{Stab}_\Gamma(y)
 =\{g\in\Gamma:g\cdot y=y\},\, \text{ for }y\in Y.
\]
The stabilizers below refer to points of the projective or
normalized Gram space before taking the quotient by the unit group.

\begin{proposition}[Compatibility between angular levels]
\label{prop:angular-levels}
The following maps between orbit spaces are well-defined and
injective:
\[
 \begin{aligned}
 \mathbb P(V_M)/A_M
 &\longrightarrow \mathbb P(V_{M+1})/A_{M+1},\\
 A_M\cdot[\theta]
 &\longmapsto A_{M+1}\cdot[P_M\theta],
 \end{aligned}
\]
and
\[
 \begin{aligned}
 \mathscr G_{m_M}/A_M
 &\longrightarrow \mathscr G_{m_{M+1}}/A_{M+1},\\
 A_M\cdot G
 &\longmapsto A_{M+1}\cdot\Phi_M(G).
 \end{aligned}
\]
The Gram map preserves rank.

For every nonzero $\theta\in V_M$,
\begin{equation}\label{eq:scalar-stabilizer-refinement}
 \operatorname{Stab}_{A_{M+1}}([P_M\theta])
 =
 \pi_M^{-1}\bigl(\operatorname{Stab}_{A_M}([\theta])\bigr).
\end{equation}
For every $G\in\mathscr G_{m_M}$,
\begin{equation}\label{eq:gram-stabilizer-refinement}
 \operatorname{Stab}_{A_{M+1}}(\Phi_M(G))
 =
 \pi_M^{-1}\bigl(\operatorname{Stab}_{A_M}(G)\bigr).
\end{equation}
\end{proposition}

\begin{proof}
Let $\one_M$ denote the constant vector on $A_M$. Then
\[
 P_M\one_M=\one_{M+1},
 \qquad
 P_M^{\mathsf T}\one_{M+1}=q\one_M.
\]
Consequently, $\Phi_M(G)$ is positive semidefinite,
annihilates $\one_{M+1}$, and satisfies
\[
 \tr\Phi_M(G)
 =q^{-1}\tr(P_M^{\mathsf T}P_MG)
 =\tr G=1.
\]
Since $P_M$ has full column rank, multiplication by $P_M$ and
$P_M^{\mathsf T}$ preserves the rank of $G$. Thus
$\Phi_M(G)\in\mathscr G_{m_{M+1}}$ with the same rank.
Both pullback constructions are injective, as is seen from
\[
 \theta=q^{-1}P_M^{\mathsf T}(P_M\theta),
 \qquad
 G=q^{-1}P_M^{\mathsf T}\Phi_M(G)P_M.
\]

For $h\in A_{M+1}$, let $P_h$ be its permutation matrix on
the finer labels, and let $P_{\pi_M(h)}$ be the corresponding
permutation matrix on the coarse labels. Pullback intertwines
these actions:
\[
 P_hP_M=P_MP_{\pi_M(h)}.
\]
Hence
\[
 P_h\Phi_M(G)P_h^{\mathsf T}
 =
 \Phi_M\bigl(P_{\pi_M(h)}GP_{\pi_M(h)}^{\mathsf T}\bigr).
\]
Surjectivity of $\pi_M$ shows that each coarse-level orbit
maps into one finer-level orbit. Conversely, an equivalence
between two pulled-back projective classes or Gram matrices
descends to the coarse level by these identities and injectivity.
This proves that the maps on orbit spaces are well-defined and
injective.

The same argument with a single class proves
\eqref{eq:scalar-stabilizer-refinement} and
\eqref{eq:gram-stabilizer-refinement}.
\end{proof}

These maps allow each level to be identified with a subset of the
next. Every locally constant angular kernel on the compact space
$\Ocal^\times$ factors through some $A_M$: a finite cover by
cosets on which the kernel is constant admits a common finer
level. The resulting unions, taken using the identifications
above, therefore parametrize nonzero locally constant angular
kernels under the respective equivalence relations.

The stabilizer identities explain how symmetries behave under
refinement. In particular, every pulled-back scalar projective
class is fixed by $\ker\pi_M$. If the original class has trivial
stabilizer, its pullback has stabilizer exactly $\ker\pi_M$,
of order $q$. Thus the pullback can represent the same angular
function on $\Ocal^\times$ while its stabilizer grows with the
ambient finite unit group.

We now determine the scalar projective fixed loci at a fixed
level. A projective class $[a]$ records a nonzero scalar function on \(A_M\)
up to nonzero real scaling. A permutation fixes $[a]$ precisely
when it preserves the line spanned by $a$.

\begin{proposition}[Projective stabilizers]\label{prop:stabilizers}
Let a finite group $\Gamma$ act on $m\geq2$ labels by
permutations. Let $P_g$ be the permutation matrix of
$g\in\Gamma$, acting on
\[
 H_m=\left\{a\in\R^m:\sum_{j=1}^m a_j=0\right\}.
\]
Write $c(g)$ for the number of cycles of this permutation,
including cycles of length one, and $e(g)$ for the number of
cycles of even length.

\begin{enumerate}
\item The locus fixed by $g$ in $\mathbb P(H_m)$ is
\[
 \operatorname{Fix}_{\mathbb P(H_m)}(g)
 =F_g^+\sqcup F_g^-,
\]
where
\[
 F_g^+=\mathbb P\bigl(\ker(P_g-I)\cap H_m\bigr),
 \qquad
 F_g^-=\mathbb P\bigl(\ker(P_g+I)\cap H_m\bigr).
\]
Here $\mathbb P(\{0\})=\emptyset$. The set $F_g^+$ is
nonempty exactly when $c(g)\geq2$, and then
\[
 \dim F_g^+=c(g)-2.
\]
The set $F_g^-$ is nonempty exactly when $e(g)\geq1$, and then
\[
 \dim F_g^-=e(g)-1.
\]

\item For the action of $A_M$ on its own elements by
multiplication, put $m=m_M\geq2$. If $g\in A_M$ has order
$\ell$, then
\[
 c(g)=\frac m\ell,
 \qquad
 e(g)=
 \begin{cases}
  m/\ell,&\ell\text{ is even},\\
  0,&\ell\text{ is odd}.
 \end{cases}
\]
The positive fixed locus is empty when $m/\ell=1$ and
otherwise has dimension $m/\ell-2$. The negative fixed locus
exists exactly when $\ell$ is even, with dimension $m/\ell-1$.

If $m\geq3$, the action on $\mathbb P(H_m)$ is faithful,
and a dense open subset consists of classes with trivial
stabilizer. If $m=2$, the projective space has one point and
the action is trivial.
\end{enumerate}
\end{proposition}

\begin{proof}
A class $[a]$ is fixed by $g$ if and only if
$P_ga=\lambda a$ for some $\lambda\in\R^\times$.
Permutation matrices preserve the Euclidean norm, so
$\lambda\in\{1,-1\}$. The two eigenspaces intersect only at
zero, which gives the disjoint union in \textup{(i)}.

A $1$-eigenvector is constant on each cycle, giving $c(g)$
independent values. The condition that its coordinates sum
to zero imposes one independent linear equation, with
coefficients equal to the cycle lengths. Therefore
\[
 \dim\bigl(\ker(P_g-I)\cap H_m\bigr)=c(g)-1.
\]
A $-1$-eigenvector alternates on each cycle. Such a vector can
be nonzero only on even cycles, and each even cycle contributes
one independent alternating vector. Every such vector already
has coordinate sum zero. Hence
\[
 \dim\bigl(\ker(P_g+I)\cap H_m\bigr)=e(g).
\]
Projectivization subtracts one from each positive vector-space
dimension, proving the stated dimensions and empty cases.

For \textup{(ii)}, every orbit of multiplication by an element
of order $\ell$ has length $\ell$. This gives the cycle counts.

Assume $m\geq3$. If an element acts trivially on
$\mathbb P(H_m)$, its restriction to $H_m$ preserves every
line and is therefore scalar. Since the permutation is
orthogonal, this scalar is $1$ or $-1$.
In the first case the permutation is the identity on both
$H_m$ and the constant vectors, so it is the identity on
$\R^m$. The multiplication action on $A_M$ is faithful,
so the element is the identity.
In the second case its matrix would be
\[
 -I+\frac2m\one\one^{\mathsf T},
\]
whose diagonal entries are negative for $m>2$, contradicting
the definition of a permutation matrix.
Thus the projective action is faithful.

Every nonidentity element consequently has a fixed locus
that is a union of proper projective linear subspaces.
The complement of the finite union of these loci is dense
and open, and every class in that complement has trivial
stabilizer. When $m=2$, the nonidentity multiplication
permutation acts as $-I$ on the one-dimensional space $H_2$,
so its projective action is trivial.
\end{proof}

\begin{remark}[Stabilizers in the Gram quotient]
\label{rem:gram-stabilizers}
For a normalized Gram matrix, the stabilizer is
\[
 \operatorname{Stab}_{A_M}(G)
 =\{u\in A_M:P_uGP_u^{\mathsf T}=G\}.
\]
For a rank-one matrix
$G=vv^{\mathsf T}/\norm{v}_2^2$, this condition is equivalent
to $P_uv=\pm v$, so \cref{prop:stabilizers} gives the same
stabilizer as for the scalar projective class $[v]$.
In higher rank, the condition concerns invariance of a
quadratic form. For example, when $m\geq3$,
\[
 G_{\mathrm{iso}}
 =\frac1{m-1}
   \left(I-\frac1m\one\one^{\mathsf T}\right)
\]
is a trace-one Gram matrix of rank $m-1$ fixed by every
permutation. It is the orthogonal projection onto $H_m$
divided by its trace.
\end{remark}
\begin{example}[Arithmetic beyond the residue cardinality]\label{ex:arithmetic}
At $q=2$ and $M=3$, the quotients \(A_3\) for \(K=\mathbb Q_2\) and \(K=\mathbb F_2((t))\) each have four elements. But their acting
groups can differ:
\[
 (\Z/8\Z)^\times\simeq C_2\times C_2,
 \qquad
 (\mathbb F_2[t]/t^3)^\times\simeq C_4.
\]
In the second group, $(1+t)^2=1+t^2$ and $(1+t)^4=1$.
The mean-zero real regular representation of $C_2\times C_2$ is a sum
of three one-dimensional characters. That of $C_4$ is a one-dimensional
sign representation plus a two-dimensional rotation representation.
The difference is visible in the projective stabilizer strata.
Every nonidentity element of $C_2\times C_2$ fixes a projective point
and a projective line in $\mathbb P(V_3)$, whereas a generator of
$C_4$ fixes only one projective point. These descriptions follow from
\cref{prop:stabilizers}.
\end{example}

\begin{remark}[Dependence on a uniformizer]\label{rem:uniformizer}
The angular realization is attached to $(K,t)$, not just to $K$.
If $t'=ut$ with $u\in\Ocal^\times$, then
$\ac_{t'}(z)=u^{-v(z)}\ac_t(z)$. This is a shell-dependent change of
direction. This change of direction generally cannot be represented by a single element of $A_M$ acting uniformly on all valuation shells. The quotient groups and finite-dimensional spaces of functions on \(A_M\) are intrinsic to the
powers of the maximal ideal, but their homogeneous lifts require this
choice. For a fixed ball kernel, taking the additional supremum over all
unit-coordinate changes makes its dilation family independent of the chosen
uniformizer, since the unit change at scale $k$ can be absorbed in that
supremum. This observation does not identify the pure angular lifts arising
from different choices of $t$.
\end{remark}

\section{Residue kernels and radial--angular decomposition}
\label{sec:residue}

For $N\geq1$, write
\[
 R_N=\Ocal/t^N\Ocal,
 \qquad
 \rho_N:\Ocal\to R_N,
 \qquad
 \rho_N(z)=z+t^N\Ocal.
\]
The fibers of $\rho_N$ partition $\Ocal$ into $q^N$ balls,
each of measure $q^{-N}$.

\begin{definition}[Residue tables and residue depth]
\label{def:residue-kernel}
Let $E$ be a finite-dimensional real normed space.
A \emph{residue table at level $N$} is a function
$v:R_N\to E$. The space of cancellative residue tables is
\[
 W_N(E)=
 \left\{v:R_N\to E:\sum_{a\in R_N}v(a)=0\right\}.
\]
A table $v\in W_N(E)$ determines the diamond kernel
\[
 \theta_{v,N}=v\circ\rho_N
 \quad\text{on }\Ocal.
\]
A diamond kernel has \emph{residue depth at most $N$} if it
equals $\theta_{v,N}$ almost everywhere for some
$v\in W_N(E)$. A kernel admitting such a representation at
some finite level is called a \emph{residue kernel}.
Its residue depth is the least admissible level. The zero
kernel is assigned depth zero.
\end{definition}

Each entry $v(a)$ specifies the kernel value on the entire
ball $\rho_N^{-1}(\{a\})$. In particular, $v(0)$ is its
value almost everywhere on $t^N\Ocal$. We use
\[
 \norm{v}_\infty=\max_{a\in R_N}\norm{v(a)}_E
 =\norm{\theta_{v,N}}_\infty.
\]
A representation at level $N$ may describe a kernel of
smaller residue depth.

For $v\in W_N(E)$, abbreviate
$T_{\theta_{v,N},k}$ by $T_{v,k}$. Given $B=B_k(x)$ and
$a\in R_N$, choose any representative $\widetilde a\in\Ocal$
and put
\[
 B_a=x+t^k\widetilde a+t^{k+N}\Ocal.
\]
This ball is independent of the chosen representative.
Splitting $B$ into these $q^N$ subballs gives
\begin{equation}\label{eq:residue-average}
 T_{v,k}f(x)
 =q^{-N}\sum_{a\in R_N}v(a)f_{B_a}.
\end{equation}
Here $k$ specifies the scale of the averaging ball, whereas
$N$ specifies its residue resolution. The underlying
partition is independent of the center chosen within $B$,
but its labels depend on that center. A fixed table can
therefore produce different measurements at different
centers in the same ball.

\begin{remark}
At a fixed center $x$ and scale $k$, every vector
$(f_{B_a})_{a\in R_N}$ is realized by an input constant on
each residue subball. Hence \cref{thm:finite-moduli}
classifies equality and mutual domination of the one-scale
norm measurements in \eqref{eq:residue-average}, with the
residue-class labels fixed. Passing to the supremum over
scales permits further identifications, as
\cref{ex:dilation-redundancy} shows.
\end{remark}

Let $U_N=R_N^\times$ act by $(u\cdot v)(a)=v(u^{-1}a)$. Reduction identifies $U_N$ with $A_N$. Here $U_N$ acts on the
entire residue ring $R_N$, including its nonunit elements.
The compact quotient
\begin{equation}\label{eq:residue-moduli}
 \mathfrak M_N^\res=\mathbb P(W_N(\R))/U_N
\end{equation}
classifies nonzero scalar tables up to scaling and unit-coordinate changes.
Equivalently, it classifies their one-scale scalar functionals under the
same equivalence, by the kernel-inclusion argument in
\cref{thm:finite-moduli}. Its dimension is $q^N-2$.
The stabilizer loci are given explicitly by \cref{prop:stabilizers}, using
the cycle decomposition of multiplication by each unit on $R_N$.
The action has a dense open locus of trivial projective stabilizer: a
nonidentity unit permutes $R_N$ nontrivially and fixes zero, and hence
cannot act as a scalar on its mean-zero subspace. For $q^N=2$ the unit
group is trivial.
The pullbacks $W_N\hookrightarrow W_{N+1}$ induce embeddings of these
quotients: reduction $U_{N+1}\to U_N$ is surjective, and any equivalence
of pulled-back tables descends. Their union parametrizes locally constant
cancellative ball kernels up to the stated table equivalence.
\Cref{ex:dilation-redundancy} explains why this is not a classification of
the full maximal operators.

Write
\[
 W_N^{\mathrm{rad}}
 =\{v\in W_N(\R):u\cdot v=v\text{ for every }u\in U_N\}.
\]
These are the radial tables. A residue table $v:R_N\to E$ is \emph{radial} if it is
constant on each orbit of $U_N$. Equivalently, $v$ is constant on each nonzero valuation shell $t^jR_N^\times$, $0\leq j<N$,
with a separate value at the zero residue class.

\begin{proposition}[Radial--angular decomposition]\label{prop:decomposition}
For scalar residue kernels, there is an orthogonal $U_N$-equivariant
decomposition
\begin{equation}\label{eq:decomposition}
 W_N(\R)=W_N^{\mathrm{rad}}\oplus
          \bigoplus_{j=0}^{N-1}V_{N-j},
 \qquad \dim W_N^{\mathrm{rad}}=N.
\end{equation}
The angular summand $V_{N-j}$ is placed on the shell
$t^j\Ocal^\times$ and $U_N$ acts on it through $A_{N-j}$.
The radial summand is fixed pointwise by $U_N$.
\end{proposition}
\begin{proof}
The residue partition decomposes as
\[
 R_N=\{0\}\sqcup\bigsqcup_{j=0}^{N-1}t^jA_{N-j}.
\]
For a table $v$, let $b_j$ be its average on $t^jA_{N-j}$ and let
$b_\infty=v(0)$. On each shell subtract $b_j$. The remainder is an
element of $V_{N-j}$. The remaining radial table is determined by the
$N+1$ numbers $b_0,\ldots,b_{N-1},b_\infty$, subject to
\[
 \sum_{j=0}^{N-1}(1-q^{-1})q^{-j}b_j+q^{-N}b_\infty=0.
\]
This gives its dimension and uniqueness. The summands on different shells
have disjoint support, and each angular remainder has mean zero on its
shell. They are therefore orthogonal for normalized counting measure on
$R_N$. Multiplication by a unit preserves shells and their averages,
giving equivariance. As a dimension check,
$N+\sum_{j=0}^{N-1}((q-1)q^{N-j-1}-1)=q^N-1$.
\end{proof}

This decomposition identifies the additional information carried by residue
tables. Angular resolution records digits after the first nonzero digit.
Residue resolution records an initial segment of the entire address,
including its valuation.

\begin{proposition}[Residue approximation of angular kernels]
\label{prop:angular-truncation}
Let $M\geq1$ be an integer and let
$\theta\in L^\infty(\Ocal^\times;E)$ have mean zero with respect to
$\nu$ and factor through $A_M$. Recall 
$\widetilde\theta(z)=\theta(\ac_t(z))$ on
$\Ocal\setminus\{0\}$. For an integer $L\geq0$, define
\[
 \theta^{[L]}(z)=\widetilde\theta(z)
                         \one_{\Ocal\setminus t^{L+1}\Ocal}(z).
\]
Then $\theta^{[L]}$ is cancellative and has residue depth at most $L+M$.
For $1\leq p<\infty$,
\[
 \norm{\widetilde\theta-\theta^{[L]}}_{L^p(\Ocal)}
 =q^{-(L+1)/p}\norm{\theta}_{L^p(\Ocal^\times,\nu)}.
\]
At each fixed center and scale, the corresponding averages converge for
every locally integrable input. A nonzero $\widetilde\theta$ never factors
through a finite residue quotient.
\end{proposition}
\begin{proof}
Each shell has angular mean zero, proving cancellation of the truncation.
On shell $j\leq L$, knowledge modulo $t^{j+M}$ determines the angular
class; the cutoff is constant on classes modulo $t^{L+1}$. Since $M\geq1$,
depth $L+M$ suffices. Summing shell measures on the omitted tail gives the
norm identity. For fixed $x,k$, the omitted integral is bounded by
\[
 q^k\norm{\theta}_\infty
       \int_{x+t^{k+L+1}\Ocal}|f(y)|\,d\mu(y),
\]
which tends to zero by absolute continuity of the integral.
If the lift were constant on residue classes of some depth $N$, it would
be constant on $t^N\Ocal$. Every angular class occurs there with positive
measure, forcing $\theta$ to be constant. Cancellation would then force
$\theta=0$.
\end{proof}

For bounded inputs the truncation gives a uniform estimate for the
full maximal operators:
\[
 \norm{\D_{\widetilde\theta}f-\D_{\theta^{[L]}}f}_\infty
 \leq q^{-(L+1)}\norm{\theta}_\infty\norm{f}_\infty.
\]
Indeed, the omitted part of the normalized ball has measure
$q^{-(L+1)}$, so the same bound holds for every center and scale
before taking the supremum. For general locally integrable inputs,
\cref{prop:angular-truncation} gives convergence at each fixed center
and scale.

\section{Radial residue operators and polynomial division}\label{sec:radial}
This section classifies radial residue operators on $L^\infty(K)$,
with the full supremum over all scales. The classification follows by
expressing radial averages as finite filters of ball-average differences.
\Cref{ex:radial-domain} explains the role of the bounded-input hypothesis.

For $k\in\Z$ define
\[
 E_kf(x)=\frac1{\mu(B_k(x))}\int_{B_k(x)}f\,d\mu,
 \qquad \Delta_kf(x)=E_{k+1}f(x)-E_kf(x).
\]
The value $E_kf(x)$ is the average on the ball of scale $k$
containing $x$. The difference $\Delta_kf(x)$ records the change
in that average when passing to the child containing $x$.
A radial kernel combines finitely many of these changes.
Let a radial kernel represented at residue level \(N\) have values $b_j$ on
$t^j\Ocal^\times$, $0\leq j<N$, and $b_N$ on $t^N\Ocal$.
Set
\[
 \alpha_0=b_0,\qquad
 \alpha_j=q^{-j}(b_j-b_{j-1})\quad(1\leq j\leq N),\qquad
 c_j=\sum_{i=j+1}^N\alpha_i\quad(0\leq j<N).
\]
Splitting each shell into the difference of two balls gives
\begin{equation}\label{eq:radial-filter}
 T_{\theta,k}f(x)=\sum_{j=0}^N\alpha_jE_{k+j}f(x)
                =\sum_{j=0}^{N-1}c_j\Delta_{k+j}f(x).
\end{equation}
Here cancellation is equivalent to $\sum_{j=0}^N\alpha_j=0$.
Conversely, every nonzero finite real coefficient list $c$ is obtained:
put $\alpha_0=-c_0$, $\alpha_j=c_{j-1}-c_j$ for $0<j<N$,
$\alpha_N=c_{N-1}$, and $b_j=\sum_{i=0}^jq^i\alpha_i$.
We write $\D_c$ for the radial operator in \eqref{eq:radial-filter} and
associate to it the polynomial $p_c(z)=\sum_jc_jz^j$.
Zero padding and index shifts are harmless, so coefficient lists may be
viewed as finitely supported sequences on $\Z$ and symbols as Laurent
polynomials.
For sequences we use
\[
 (a*b)_j=\sum_{n\in\Z}a_n b_{j-n}.
\]
If $a\in\ell^1(\Z;\R)$, its symbol is the absolutely convergent series
\[
 \widehat a(z)=\sum_{n\in\Z}a_nz^n,\qquad |z|=1.
\]
Thus $c=a*b$ is equivalent to $p_c=\widehat a\,p_b$ on the unit circle.

We write $c_0(\Z)$ for the real sequences tending to zero as
$|j|\to\infty$. The duality between $\ell^1(\Z)$ and $c_0(\Z)$
uses the pairing $\sum_j a_jd_j$; the weak-star topology below is
$\sigma(\ell^1,c_0)$.
\begin{lemma}[Finite-filter domination]\label{lem:filter}
For a nonzero finitely supported real sequence $b$ and a finitely supported
real sequence $c$, the least constant in
\begin{equation}\label{eq:filter-comparison}
 \sup_k\left|\sum_jc_jd_{k+j}\right|
 \leq C\sup_k\left|\sum_jb_jd_{k+j}\right|
 \qquad(d\in\ell^\infty(\Z))
\end{equation}
is $\norm a_{\ell^1}$ if an $a\in\ell^1(\Z)$ satisfies $c=a*b$,
and is infinite if no such $a$ exists. When it exists, $a$ is unique.
It suffices to test \eqref{eq:filter-comparison} on finitely supported $d$.
\end{lemma}
\begin{proof}
If $c=a*b$, absolute convergence permits rearrangement for bounded $d$:
\[
 \sum_jc_jd_{k+j}=\sum_na_n\sum_jb_jd_{k+n+j}.
\]
This proves the inequality with constant $\norm a_1$.
Conversely, suppose the inequality holds for finitely supported $d$.
Uniform approximation extends it to $d\in c_0(\Z)$.
In the dual pair $(\ell^1,c_0)$, the set
\[
 K_C=\{a*b:\norm a_1\leq C\}
\]
is weak-star compact and convex: the $\ell^1$ ball is weak-star compact,
and convolution by finitely supported $b$ is weak-star continuous because
its preadjoint preserves $c_0$. Indeed, the preadjoint of convolution by $b$ maps $d$ to the sequence $\left(\sum_j b_jd_{n+j}\right)_{n\in\Z},$ which belongs to $c_0(\Z)$ whenever $d$ does. Its support function at $d\in c_0$ is
$C\sup_k|\sum_jb_jd_{k+j}|$. The assumed inequality at $k=0$ and separation
therefore imply $c\in K_C$.
If $a*b=0$, the absolutely convergent Fourier series of $a$ multiplied by
the nonzero Laurent polynomial $p_b$ is zero on the unit circle. The latter
polynomial has finitely many zeros there. Continuity forces the Fourier
series of $a$ to vanish everywhere. Uniqueness of Fourier coefficients
gives $a=0$. This proves uniqueness and the optimal-constant assertion.
\end{proof}
By the unit-circle zero divisor of a nonzero Laurent polynomial
we mean its zeros on $\{z\in\C:|z|=1\}$, counted with multiplicity.
Equality of these divisors means equality of the zero locations
and their multiplicities.

A finite real filter is a finitely supported sequence $c=(c_j)_{j\in\Z}$.
\begin{theorem}[Classification of radial residue operators on bounded inputs]
\label{thm:radial-classification}
Let $b,c$ be nonzero finite real filters of radial residue kernels.
The following are equivalent:
\begin{enumerate}
\item $\D_cf(x)\leq C\D_bf(x)$ for every bounded input and every $x$, for
some finite constant $C$;
\item $c=a*b$ for a sequence $a\in\ell^1(\Z)$;
\item every unit-circle zero of $p_b$ is a zero of $p_c$ with at least the
same multiplicity.
\end{enumerate}
The least comparison constant is $\norm a_1$. Mutual domination on bounded
inputs is classified by the unit-circle zero divisor of the filter
polynomial. Exact equality on bounded inputs holds precisely when
\begin{equation}\label{eq:exact-filter}
 c_j=\varepsilon b_{j-m}\quad\text{for some }m\in\Z,
 \quad\varepsilon\in\{1,-1\}.
\end{equation}
The same statements hold when comparisons are required only almost
everywhere for each bounded input.
\end{theorem}
\begin{proof}
For bounded $f$, the sequence $d_k=\Delta_kf(x)$ is bounded.
Thus \cref{lem:filter} proves \textup{(ii)}$\Rightarrow$\textup{(i)}.
For the converse, every finitely supported real sequence $d_k$ can be
realized as $\Delta_kf(0)$ for a bounded, locally constant, compactly
supported radial function. To see this, choose integers $L<U$ containing
its support and prescribe $m_L=0$,
$m_{k+1}=m_k+d_k$ for $L\leq k\leq U$.
Put $f=0$ outside $t^L\Ocal$, put $f=m_{U+1}$ on $t^{U+1}\Ocal$, and
give it the shell value
\[
 \frac{q m_k-m_{k+1}}{q-1}
 \quad\text{on }t^k\Ocal^\times,\quad L\leq k\leq U.
\]
If $E_{k+1}f(0)=m_{k+1}$, averaging over the shell and the inner ball gives
\[
 E_kf(0)
 =(1-q^{-1})\frac{qm_k-m_{k+1}}{q-1}
   +q^{-1}m_{k+1}
 =m_k.
\]
Starting from the innermost ball proves
$E_kf(0)=m_k$ for $L\leq k\leq U+1$.
Since $m_L=0$, the integral of $f$ on its support is zero, so
$E_kf(0)=0$ for $k<L$. For $k\geq U+1$ the average is $m_{U+1}$.
Hence $\Delta_kf(0)=d_k$ at every scale. Applying
\textup{(i)} at zero and then \cref{lem:filter} proves
\textup{(ii)}, with the same least constant.

If $c=a*b$, then on the unit circle $p_c=\widehat a\,p_b$, where
$\widehat a$ is continuous. A zero of $p_b$ of higher order than the
corresponding zero of $p_c$ would make $p_c/p_b$ unbounded there.
This proves \textup{(iii)}. Conversely, under \textup{(iii)} all possible
poles of the rational function $p_c/p_b$ on the unit circle are removable.
It is holomorphic on an annulus containing that circle. Cauchy estimates
on a smaller closed annulus imply exponential decay of its Laurent
coefficients in both directions. Since $p_b$ and $p_c$ have real
coefficients, their quotient $R=p_c/p_b$ satisfies
$R(\bar z)=\overline{R(z)}$. Uniqueness of the Laurent expansion
therefore makes its coefficients real. They form an $\ell^1$ sequence
$a$, and multiplication of the absolutely convergent series gives $c=a*b$.

Mutual domination now means equality of the zero divisors. If the maximal
operators are exactly equal, there exist $a,e\in\ell^1$ with
$c=a*b$, $b=e*c$, and $\norm a_1,\norm e_1\leq1$.
Uniqueness in \cref{lem:filter} gives $a*e=\delta_0$. The sequence \(\delta_0\) is equal to one at index zero and zero elsewhere.
The inequality $|a*e|\leq|a|*|e|$ has equality in total mass, since
$1\leq\norm a_1\norm e_1\leq1$. Therefore $|a|*|e|=\delta_0$.
Every sum of a point in the support of $a$ and a point in the support of
$e$ must be zero. Both supports are singletons, and their coefficients
have absolute value one. This gives \eqref{eq:exact-filter}, whose
sufficiency is immediate from the supremum over $k$.

For almost-everywhere comparisons, the realizing functions above are
uniformly continuous. The bound by their translation modulus, used in
\cref{thm:angular-rigidity}, makes both maximal functions continuous.
Hence their comparison extends to zero, and the same proof applies.
\end{proof}

The division step takes place in the Wiener algebra of absolutely
convergent Fourier series \cite{Wiener1932}. If $p_b$ has no zeros
on the unit circle, Wiener's theorem gives an absolutely summable
inverse. For the polynomials considered here, cancellation of all
unit-circle poles makes $p_c/p_b$ holomorphic on an annulus and yields
the stronger exponential decay used in the proof. The realization
of finitely supoorted sequences by radial inputs transfers this division
criterion, including its optimal constant, to diamond operators.

Thus the radial comparability moduli can be identified with finite effective
divisors on the unit circle invariant under complex conjugation. Every such
divisor occurs, by taking the real polynomial with those roots. Positive
scaling and a shift of the filter do not affect this invariant, whereas
zeros off the unit circle affect optimal constants and exact equality.
For example, if $0<\rho<1$, then $p_b(z)=1-\rho z$ has the inverse
$\sum_{n\geq0}\rho^nz^n$ and the optimal bound is
\[
 \D_1f\leq\frac1{1-\rho}\D_bf\qquad(f\in L^\infty(K)).
\]
Here $\D_1=\sup_k|\Delta_kf|$. A filter $1+z$ has a zero at $-1$;
its operator is dominated by $2\D_1$, but no reverse bound holds on bounded
inputs.

\begin{example}[Why bounded inputs are specified]\label{ex:radial-domain}
Fix $\beta>0$ and let $f(y)=|y|^\beta$. This function is locally integrable.
On every ball $B_k(x)$ containing zero,
\[
 E_kf(x)=A_\beta q^{-\beta k},\qquad
 A_\beta=\frac{1-q^{-1}}{1-q^{-(\beta+1)}}.
\]
For all sufficiently negative $k$ both consecutive balls contain zero,
so $\Delta_kf(x)$ is a nonzero geometric sequence with ratio $q^{-\beta}$.
Consequently $\D_1f(x)=\infty$ at every $x$.
The filter $b(z)=1-q^\beta z$ annihilates these coarse-scale differences.
For $x\ne0$, all sufficiently small balls lie on a shell where $f$ is
constant, so the fine-scale differences vanish as well. Only finitely many
transition scales remain, giving $\D_bf(x)<\infty$. At $x=0$ it vanishes
at every scale. Both filter polynomials have empty unit-circle zero
divisors. Their bounded-input comparability therefore cannot be extended
to all of $L^1_{\loc}(K)$.
\end{example}

\section{Cancellative envelopes and best constant approximation}
\label{sec:envelope}
\subsection{A probability-space identity}
Bounded cancellative tests annihilate constants. Their natural
dual quantity is therefore the best $L^1$ approximation by a
constant. On a probability space, a median of $h$ is a real
number $c$ satisfying
\[
 \nu\{h<c\}\leq\tfrac12,
 \qquad
 \nu\{h>c\}\leq\tfrac12.
\]
The following identity identifies the cancellative envelope
and constructs tests attaining its value.

\begin{theorem}[Median duality]\label{thm:median}
Let $(Y,\mathcal A,\nu)$ be a probability space and $h\in L^1(Y;\R)$.
Then
\begin{equation}\label{eq:median-duality}
 \sup_{\substack{\eta\in L^\infty(Y)\\|\eta|\leq1,\ \int\eta\,d\nu=0}}
       \left|\int_Y\eta h\,d\nu\right|
 =\inf_{c\in\R}\int_Y|h-c|\,d\nu.
\end{equation}
The infimum is attained at any median of $h$, and the supremum is attained.
Writing $\bar h=\int h\,d\nu$ and $\delta(h)=\inf_c\int|h-c|\,d\nu$,
\begin{equation}\label{eq:median-mean}
 \delta(h)\leq\int|h-\bar h|\,d\nu\leq2\delta(h).
\end{equation}
If $Y$ has $m\geq2$ equally weighted atoms, the optimal second constant is
$2(1-1/m)$.
\end{theorem}
\begin{proof}
Choose a median $c$, so that $p_+=\nu\{h>c\}\leq1/2$ and
$p_-=\nu\{h<c\}\leq1/2$. Such a $c$ exists from the distribution
function of $h$. Put $p_0=\nu\{h=c\}$. Define $\eta=1$ on $\{h>c\}$
and $\eta=-1$ on $\{h<c\}$. If $p_0>0$, set
$\eta=(p_--p_+)/p_0$ on $\{h=c\}$. Its absolute value is at most one
because $|p_+-p_-|\leq p_0$. If $p_0=0$, then $p_+=p_-=1/2$.
In either case $\int\eta=0$ and
\[
 \int\eta h=\int\eta(h-c)=\int|h-c|.
\]
For any admissible test and any constant $b$,
$|\int\eta h|=|\int\eta(h-b)|\leq\int|h-b|$.
This proves optimality of the median and both attainment assertions.

Taking $c=\bar h$ proves the first inequality in
\eqref{eq:median-mean}. The triangle inequality and
$|\bar h-c|\leq\int|h-c|$ prove the second.
For equal atoms with values $h_i$, a sharper estimate is
\begin{align*}
 \frac1m\sum_i|h_i-\bar h|
 \leq\frac1{m^2}\sum_{i,j}|h_i-h_j|
 \leq\frac{2(m-1)}{m^2}\sum_i|h_i-c|.
\end{align*}
Choose a median $c$. A list with one entry equal to one and all remaining
entries equal to zero attains the ratio $2(1-1/m)$.
\end{proof}

For the full ball envelope define
\begin{equation}\label{eq:full-envelope}
 \D^{\mathrm{all}}f(x)=
 \sup_s\sup_{\substack{\eta\in L^\infty(B_s(x))\\
                        |\eta|\leq1,\ \eta_{B_s(x)}=0}}
      \left|\frac1{b_s}\int_{B_s(x)}\eta f\,d\mu\right|.
\end{equation}
An affine identification with $B_0$ identifies these tests with all
normalized scalar diamond kernels. The order of the two suprema is
immaterial.

\begin{corollary}\label{cor:full-envelope}
For every $f\in L^1_{\loc}(X)$,
\begin{equation}\label{eq:sharp-comparison}
 \D^{\mathrm{all}}f(x)=\sup_s\delta_{B_s(x)}(f),
 \qquad
 \tfrac12\M^\sharp f\leq\D^{\mathrm{all}}f\leq\M^\sharp f.
\end{equation}
Consequently, on any solid Banach function lattice
$\mathcal X\subseteq L^1_{\loc}(X)$, boundedness of
$\D^{\mathrm{all}}:\mathcal X\to\mathcal X$ is equivalent to boundedness
of $\M^\sharp:\mathcal X\to\mathcal X$.
\end{corollary}
\begin{proof}
Apply \cref{thm:median} with normalized measure on each ball, then take
the supremum. Solidity means that $|g|\leq|h|$ almost everywhere and
$h\in\mathcal X$ imply $g\in\mathcal X$ and
$\|g\|_{\mathcal X}\leq\|h\|_{\mathcal X}$.
Applying this property to the pointwise comparisons proves the
boundedness assertion.
\end{proof}

For $K=\mathbb Q_p$, the comparison with sharp oscillation appears in
\cite[Lemma~5.2(c)]{Kim2009}, and its BMO seminorm version appears in
\cite[Lemma~5.5(a)]{Kim2009}. The median formulation above identifies the
cancellative envelope, constructs tests attaining its value, and gives
the optimal constants for equal finite partitions.

The constants in \eqref{eq:median-mean} are sharp on nonatomic probability
spaces. For an indicator of a set of measure $p\leq1/2$,
$\delta=p$ and mean oscillation is $2p(1-p)$. The ratio tends to two as
$p\downarrow0$, and equals one at $p=1/2$. The same constants in
\eqref{eq:sharp-comparison} cannot be improved uniformly over local fields.
For the lower constant, choose a deep subball $A\subset B$ and evaluate
$f=\one_A$ on $B$ outside the immediate child containing $A$. The
maximizing ball for both expressions is $B$, with
$p=\mu(A)/\mu(B)\to0$. For the upper constant, a measurable half-measure
subset of $B$ gives both expressions equal to $1/2$ throughout $B$.
Existence of such a subset follows from nonatomicity of Haar measure.

\subsection{Residue approximation with an error estimate}
For a local-field ball $B$, let $E_N^Bf$ equal $f_J$ on each depth-$N$
descendant $J$ of $B$. This is conditional averaging over the partition of $B$ into its
depth-$N$ descendants. It preserves the mean and is contractive:
\[
 (E_N^Bf)_B=f_B,
 \qquad
 \|E_N^Bf\|_{L^p(B)}\leq\|f\|_{L^p(B)}
 \quad(p=1,2).
\]
Define
\[
 \delta_{B,N}(f)=
 \sup_{\substack{v\in W_N(\R)\\\norm v_\infty\leq1}}
       |T_{v,k}f(x)|\quad(B=B_k(x)),
 \qquad
 \D_N^\res f(x)=\sup_k\delta_{B_k(x),N}(f).
\]
Although a fixed residue measurement depends on the address labels, this
envelope is independent of them.

\begin{theorem}[Quantitative residue approximation]\label{thm:residue-limit}
For every ball $B$ and $f\in L^1(B)$,
\begin{align}
 \delta_{B,N}(f)&=\delta_B(E_N^Bf),\label{eq:finite-median}\\
 0\leq\delta_B(f)-\delta_{B,N}(f)
 &\leq\frac1{\mu(B)}\int_B|f-E_N^Bf|\,d\mu.\label{eq:residue-error}
\end{align}
The left-hand side decreases to zero as $N\to\infty$. Consequently,
\begin{equation}\label{eq:monotone-max}
 \D_N^\res f(x)\longrightarrow\D^{\mathrm{all}}f(x)
 \quad\text{for every }f\in L^1_{\loc}(K)\text{ and every }x\in K.
\end{equation}
At level $N$, the exact comparison with the oscillation of $E_N^Bf$ is
\begin{equation}\label{eq:finite-constant}
 \frac{q^N}{2(q^N-1)}\osc_B(E_N^Bf)
 \leq\delta_{B,N}(f)\leq\osc_B(E_N^Bf),
\end{equation}
with optimal lower coefficient.
\end{theorem}
\begin{proof}
Formula \eqref{eq:residue-average} and finite median duality prove
\eqref{eq:finite-median}. The admissible tests at level $N$ are the mean-zero scalar
residue kernels of depth at most $N$ with $L^\infty$ norm
at most one. Every such test remains admissible at level
$N+1$, so $\delta_{B,N}(f)$ is nondecreasing in $N$.
Since each test is admissible for $\delta_B(f)$,
we have $\delta_{B,N}(f)\leq\delta_B(f)$. Alternatively, conditional averaging and the triangle
inequality give $\delta_B(E_N^Bf)\leq\delta_B(f)$.
The functional $\delta_B$ is $1$-Lipschitz for normalized $L^1$ distance:
\[
 |\delta_B(f)-\delta_B(g)|\leq\frac1{\mu(B)}\int_B|f-g|\,d\mu.
\]
This proves the error bound.

Locally constant functions on $B$ are dense in $L^1(B)$. One proof first
approximates by continuous functions on the compact ball, using regularity
of Haar measure, and then uses uniform continuity and finite partitions
into sufficiently small balls. If $g$ is constant at depth $N_0$, then
$E_N^Bg=g$ for $N\geq N_0$. Contractivity of $E_N^B$ gives
$\norm{f-E_N^Bf}_1\leq2\norm{f-g}_1$, proving convergence.
Thus $\sup_N\delta_{B,N}(f)=\delta_B(f)$.
Interchanging the two suprema gives \eqref{eq:monotone-max}.
Finally \eqref{eq:finite-constant} is the $m=q^N$ case of
\cref{thm:median}.
\end{proof}

If $f$ is $\alpha$-H\"older on $B$ with constant $H$, then
\[
 0\leq\delta_B(f)-\delta_{B,N}(f)
 \leq H\big(q^{-N}\operatorname{diam}B\big)^\alpha.
\]
Indeed, averaging $|f(y)-f(z)|$ over the descendant containing $y$ gives
that pointwise bound on $|f-E_N^Bf|$. The estimate records the dependence on the diameter of the base ball
$B$.

\subsection{Angular envelopes and their exact blind spot}
At a fixed center and scale let
\[
 \delta^\ang_{x,s}(f)=
 \sup_{\substack{\theta\in L^\infty(\Omega;\R)\\
                         |\theta|\leq1,\ \int\theta=0}}
           |L_\theta(P_{x,s}f)|,
 \qquad
 \D^{\ang,\mathrm{all}}f(x)=\sup_s\delta^\ang_{x,s}(f).
\]
\Cref{thm:median} gives the exact formula
\begin{equation}\label{eq:angular-envelope}
 \delta^\ang_{x,s}(f)
   =\inf_{c\in\R}\int_\Omega|P_{x,s}f-c|\,d\nu
   \leq\delta_{B_s(x)}(f).
\end{equation}
The inequality follows either by conditional averaging in the radial
variable or by viewing each angular test as a test on the full ball.

Over $K$, if $Q_M$ averages functions on $\Ocal^\times$ over cosets of
$1+t^M\Ocal$, the level-$M$ angular envelope is exactly
$\inf_c\int|Q_MP_{x,k}f-c|d\nu$. These quantities increase to
\eqref{eq:angular-envelope}, by the same density proof as in
\cref{thm:residue-limit}. Thus increasing angular resolution recovers the full angular
envelope, which measures the $L^1$ distance of the radial profile
from constants. Radial averaging occurs before angular testing,
so this envelope can leave oscillation within the ball undetected.

\begin{proposition}[Failure of radial recovery]\label{prop:radial-blindness}
In both geometric models there is no finite $C$ such that
$\M^\sharp f\leq C\D^{\ang,\mathrm{all}}f$ almost everywhere for every
locally integrable input. Over a local field the right-hand side can vanish
on a ball where the left-hand side is strictly positive.
\end{proposition}
\begin{proof}
Over $K$ take $f=\one_{t\Ocal}$. If $x\in t\Ocal$, then $f$ is constant
on every shell around $x$: it equals one for $|y-x|\leq q^{-1}$ and zero
otherwise. Every angular integral therefore vanishes. On the other hand,
\[
 \M^\sharp f(x)\geq\osc_{\Ocal}(f)=2q^{-1}(1-q^{-1})>0.
\]
In fact $\D^{\mathrm{all}}f(x)=q^{-1}$ on $t\Ocal$.

In Euclidean space choose a nonconstant smooth compactly supported radial
function $f$. Every angular average at the origin vanishes.
If $L$ is a global Lipschitz constant for $f$, then for every normalized
angular test and every radius,
\[
 |T_{\theta,r}f(x)|
 =|T_{\theta,r}f(x)-T_{\theta,r}f(0)|\leq L|x|.
\]
Choose $r_0$ with $\osc_{B(0,r_0)}(f)=a>0$.
Continuity of this fixed-radius oscillation in the center gives
$\M^\sharp f(x)\geq a/2$ in a neighborhood of zero.
For any proposed $C$, a smaller positive-measure neighborhood satisfies
$CL|x|<a/2$, contradicting the proposed inequality.
\end{proof}

\section{An exact obstruction at fixed residue resolution}
\label{sec:obstruction}
We construct inputs whose residue envelopes at a prescribed resolution and full
oscillation can be computed from independent digits. Throughout
this section, $\E$ denotes expectation on the probability space
$(\Ocal,\mu|_{\Ocal})$.

Fix representatives of $\Ocal/t\Ocal$, so each $z\in\Ocal$ has a unique
expansion $z=\sum_{j=0}^\infty d_j(z)t^j$. Under normalized Haar measure
the digits are independent and uniform: a specification of the first $L$
digits is a ball of measure $q^{-L}$.
Choose a real-valued function $\varphi$ on these $q$
representatives, taking the values $1$ and $-1$ on equally
many representatives and the value zero on the remaining
representative when $q$ is odd. We refer to $\varphi$ as a digit function. Thus
\[
 \E\varphi=0,\quad |\varphi|\leq1,\quad
 \sigma^2=\E\varphi^2=
 \begin{cases}1,&q\text{ even},\\(q-1)/q,&q\text{ odd}.
 \end{cases}
\]
We take $\sigma=(\E\varphi^2)^{1/2}>0$.
Let $S_L=\sum_{j=0}^{L-1}\varphi(d_j)$, put $a_L=\E|S_L|$, $a_0=0$,
and extend $f_L=S_L$ on $\Ocal$ by zero to $K$.

\begin{theorem}[Finite-resolution obstruction]\label{thm:resolution-obstruction}
For every $L,N\geq1$ and every $x\in\Ocal$,
\begin{equation}\label{eq:exact-obstruction}
 \D_N^\res f_L(x)=a_{\min(N,L)},
 \qquad
 \D^{\mathrm{all}}f_L(x)=\M^\sharp f_L(x)=a_L.
\end{equation}
Moreover
\begin{equation}\label{eq:sqrt}
 \sigma\sqrt{L/3}\leq a_L\leq\sigma\sqrt L.
\end{equation}
Consequently, for fixed $N$ there is no constant $C_N$ with
$\M^\sharp f\leq C_N\D_N^\res f$ on all inputs, even almost everywhere.
For $L\geq N$, the ratio on $\Ocal$ is at least $\sqrt{L/(3N)}$.
The same failure holds for every uniformly bounded family of kernels whose
residue depths have a common finite bound.
\end{theorem}
\begin{proof}
Symmetry makes zero a median of $S_j$, so its best $L^1$ distance to a
constant equals $a_j$. Conditional Jensen gives $a_{j+1}\geq a_j$.
If $B=B_k(x)\subseteq\Ocal$ with $0\leq k<L$, the first $k$ digits are
fixed and the remaining $L-k$ digits are independent. Put $c=\sum_{j=0}^{k-1}\varphi(d_j(x))$.
With normalized Haar measure on $B$, the functions $f_L$ and
$E_N^Bf_L$ have the respective distributions
\[
 c+S_{L-k},
 \qquad
 c+S_{\min(N,L-k)}.
\]
For $k\geq L$, both are constant. It follows that the largest finite-depth
distance on balls contained in $\Ocal$ is $a_{\min(N,L)}$, attained at
$B=\Ocal$. The full distance and mean oscillation are $a_L$ there.

It remains to check larger balls, which is essential for the full maximal
operator. If $B=B_k(x)\supset\Ocal$ with $k<0$, then
$\mu(\Ocal)/\mu(B)=q^k$ and $(f_L)_B=0$.
The full distribution consists of a symmetric copy of $S_L$ with weight
$q^k$ and an atom at zero. Its median distance and mean oscillation are
$q^ka_L\leq a_L$.
If $k+N\leq0$, every depth-$N$ average is zero.
If $k+N>0$, the conditional average equals
$S_{\min(k+N,L)}$ on $\Ocal$ and zero elsewhere. Its median is zero,
so
\[
 \delta_{B,N}(f_L)=q^k a_{\min(k+N,L)}
                \leq a_{\min(N,L)}.
\]
This proves both identities, including all ancestor scales.

The upper estimate in \eqref{eq:sqrt} is Cauchy--Schwarz.
Independence gives
\[
 \E S_L^2=L\sigma^2,\qquad
 \E S_L^4=L\sigma^2+3L(L-1)\sigma^4
          \leq3L^2\sigma^4,
\]
since $\sigma^2\geq2/3$. H\"older's inequality, applied to
$|S_L|^{2/3}|S_L|^{4/3}$, yields
\[
 \E S_L^2\leq(\E|S_L|)^{2/3}(\E S_L^4)^{1/3}.
\]
Rearrangement proves the lower bound. The claimed ratio follows.
Every scalar normalized depth-at-most-$N$ measurement is bounded by
$\D_N^\res$. For a vector-valued kernel with
$\norm\theta_\infty\leq A$, scalarizing by its dual unit ball gives
$\D_\theta\leq A\D_N^\res$. Taking a supremum over a uniformly bounded
family proves the final assertion.
\end{proof}

The obstruction holds pointwise on $\Ocal$, a set of positive measure.
Kova{\v c}'s Euclidean result establishes a boundedness separation
in a specified translation-invariant Banach space
\cite{Kovac2026BanachDistinguishes}. Establishing such a norm separation
for the residue families considered here is a further question.

\section{Fourier classification of quadratic diamond energy}\label{sec:energy}
The preceding obstruction concerns a supremum over scales.
Quadratic accumulation retains information across scales through
the orthogonal decomposition of nested ball averages. This connects
the construction with the classical BMO--Carleson framework
\cite{FeffermanStein1972}. Over $\mathbb Q_p^d$, Kim characterizes
Carleson measures through embeddings of maximal and positive convolution
operators \cite[Theorem~4.7]{Kim2009}. The energy considered here is
generated by translates of a fixed mean-zero residue kernel. Its finite
Fourier weights classify the energy and give the optimal constants for
BMO recovery.

\subsection{One-level measurements and their spectra}
Let $\mathbb k=\Ocal/t\Ocal$, viewed as an additive finite group of
order $q$, and equip $\R^r$ with its Euclidean norm. Let
$v:\mathbb k\to\R^r$ satisfy $\sum_a v(a)=0$.
For a ball $I=c+t^k\Ocal$, label its children by
$I_a=c+t^ka+t^{k+1}\Ocal$, and write $F(a)=f_{I_a}$.
If $x\in I_b$, then
\begin{equation}\label{eq:translated-measurement}
 T_{v,k}f(x)=H_vF(b):=\frac1q\sum_{a\in\mathbb k}v(a-b)F(a).
\end{equation}
Changing $c$ relabels the children by a common additive translation and
does not change the following energy:
\begin{equation}\label{eq:local-energy}
 e_v(I;f)=\frac1{\mu(I)}\int_I\norm{T_{v,k}f(x)}^2\,d\mu(x)
         =\frac1q\sum_{b\in\mathbb k}\norm{H_vF(b)}^2.
\end{equation}
Translation of the child labels is diagonalized by additive
characters. An additive character of $\mathbb k$ is a homomorphism
\[
 \chi:(\mathbb k,+)\longrightarrow\{z\in\C:|z|=1\}.
\]
Write $\widehat{\mathbb k}$ for the character group and $1$ for
the trivial character. We use the normalized Fourier transform
\[
 \widehat F(\chi)=\frac1q\sum_a F(a)\overline{\chi(a)},\qquad
 \widehat v(\chi)=\frac1q\sum_a v(a)\overline{\chi(a)},\qquad
 w_v(\chi)=\norm{\widehat v(\chi)}_{\C^r}^{2}.
\]
Here $\C^r$ has the usual complexification of the Euclidean norm.
Then $w_v(1)=0$ and $w_v(\bar\chi)=w_v(\chi)$.

\begin{proposition}[Spectral identity]\label{prop:fourier-energy}
For every ball and every input with finite child averages,
\begin{equation}\label{eq:spectral-energy}
 e_v(I;f)=\sum_{\chi\ne1}w_v(\chi)|\widehat F(\chi)|^2.
\end{equation}
Thus the weight function $w_v$ classifies the one-level energy exactly.
The optimal comparison constants between two such energies are determined
by the coordinatewise comparison of their weights.
\end{proposition}
\begin{proof}
Taking the Fourier transform of \eqref{eq:translated-measurement} in $b$
and substituting $z=a-b$ gives
$\widehat{H_vF}(\chi)=\widehat v(\bar\chi)\widehat F(\chi)$.
Parseval's identity proves \eqref{eq:spectral-energy}.
For necessity of any weight comparison, use a real character if
$\chi=\bar\chi$, and otherwise the real part of $\chi$.
The latter has Fourier support exactly on $\{\chi,\bar\chi\}$,
where the two weights are equal. Sufficiency follows termwise.
\end{proof}

\subsection{The Carleson functional and BMO recovery}
For $f\in L^2_{\loc}(K)$ define
\begin{align*}
 \mathcal C_v(f)
 =\sup_B\left(\frac1{\mu(B)}
       \sum_{I\subseteq B}\mu(I)e_v(I;f)\right)^{1/2},\quad
 \norm f_{\BMO_2}
 =\sup_B\left(\frac1{\mu(B)}\int_B|f-f_B|^2\,d\mu\right)^{1/2}.
\end{align*}
The sum runs over $B$ and all its descendant balls. The summands are
nonnegative, so the definition is unambiguous in $[0,\infty]$.
We use quadratic mean oscillation because the orthogonal decomposition
of ball averages is an $L^2$ identity. For $K=\mathbb Q_p$, this
seminorm is equivalent to the usual $\BMO_1$ seminorm by the
John--Nirenberg inequality and its moment consequence
\cite[Theorem~5.16 and Corollary~5.17]{Kim2009}.

The connection with Carleson measures has an explicit form. Let
$\delta_a$ denote the Dirac measure at $a\in K$, and define a positive
Borel measure on $K\times K$ by
\[
 d\rho_{v,f}(x,h)=\sum_{k\in\Z}
       \norm{T_{v,k}f(x)}^2\,d\mu(x)\,d\delta_{t^k}(h).
\]
For $B=B_n(c)$ put $\widehat B=B\times t^n\Ocal$. Since
$t^k\in t^n\Ocal$ exactly when $k\geq n$, Tonelli's theorem and
partitioning $B$ into its balls of scale $k$ give
\begin{align*}
 \rho_{v,f}(\widehat B)
 =\sum_{k=n}^{\infty}\int_B\norm{T_{v,k}f(x)}^2\,d\mu(x)
 =\sum_{I\subseteq B}\mu(I)e_v(I;f).
\end{align*}
Consequently
\[
 \sup_B\frac{\rho_{v,f}(\widehat B)}{\mu(B)}=\mathcal C_v(f)^2.
\]
If this quantity is finite, every compact subset of $K\times K$ is
contained in such a box, so $\rho_{v,f}$ is locally finite and Radon.
For $K=\mathbb Q_p$, these boxes and this Carleson norm are precisely
those of \cite[Definition~4.1, with $d=1$]{Kim2009}.

\begin{theorem}[Energy reconstruction and classification]\label{thm:energy}
Let
$\lambda_-(v)=\min_{\chi\ne1}w_v(\chi)$ and
$\lambda_+(v)=\max_{\chi\ne1}w_v(\chi)$.
Then
\begin{equation}\label{eq:energy-bmo}
 \sqrt{\lambda_-(v)}\norm f_{\BMO_2}
 \leq\mathcal C_v(f)
 \leq\sqrt{\lambda_+(v)}\norm f_{\BMO_2}.
\end{equation}
If $\lambda_-(v)=0$, the lower bound means $0\leq\mathcal C_v(f)$;
if $v=0$, the functional vanishes identically.
The displayed constants are optimal. In particular, $\mathcal C_v$ is
equivalent to $\BMO_2$ if and only if
\begin{equation}\label{eq:nondegenerate-spectrum}
 \widehat v(\chi)\ne0\quad\text{for every }\chi\ne1.
\end{equation}
Let $w:\mathbb k\to\R^{r'}$ be another mean-zero kernel, with
Euclidean norm on its output space.
More generally, for $C\geq0$,
\begin{equation}\label{eq:energy-comparison}
 \mathcal C_v(f)\leq C\mathcal C_w(f)\text{ for all }f\in L^2_{\loc}(K)
 \quad\Longleftrightarrow\quad
 w_v(\chi)\leq C^2w_w(\chi)\text{ for every }\chi\ne1.
\end{equation}
Consequently equality of the full Carleson functionals is equivalent to
equality of their Fourier weights.
\end{theorem}
\begin{proof}
Let $d_If$ be the function supported on $I$ with value $f_J-f_I$ on
each child $J$ of $I$. Parseval on the child space gives
\[
 \frac1{\mu(I)}\norm{d_If}_{L^2(I)}^2
 =\frac1q\sum_a|F(a)-f_I|^2
 =\sum_{\chi\ne1}|\widehat F(\chi)|^2.
\]
Thus \cref{prop:fourier-energy} bounds $\mu(I)e_v(I;f)$ between
$\lambda_-(v)\norm{d_If}_2^2$ and
$\lambda_+(v)\norm{d_If}_2^2$.
The differences $d_If$ for $I\subseteq B$ are orthogonal: two disjoint
supports give zero pairing, and on a proper descendant the coarser
difference is constant while the finer difference has integral zero.
For every $N\geq1$, the finite sum satisfies
\[
 \sum_{j=0}^{N-1}
 \ \sum_{\substack{I\subseteq B\\
                   \mu(I)=q^{-j}\mu(B)}}d_If
 =E_N^Bf-f_B
 \qquad\text{on }B.
\]
Density of locally constant
functions and $L^2$ contractivity of conditional averaging give convergence
to $f-f_B$ in $L^2(B)$. Hence
\begin{equation}\label{eq:haar-energy}
 \sum_{I\subseteq B}\norm{d_If}_2^2=\int_B|f-f_B|^2\,d\mu.
\end{equation}
Summing and taking the supremum proves \eqref{eq:energy-bmo}.
The sufficiency in \eqref{eq:energy-comparison} follows in the same way
from the termwise spectral comparison.

For necessity and optimality, take a mean-zero child profile $F$ on a fixed
ball $B$, extend it constantly over each child, and put it equal to zero
outside $B$. All its differences vanish except $d_Bf=f$.
Measurements on strict descendants vanish by cancellation. Measurements
on strict ancestors vanish because $f$ has mean zero on the one child
containing its support. Therefore
\[
 \mathcal C_v(f)^2=e_v(B;f),\qquad
 \norm f_{\BMO_2}^2=\frac1q\sum_a|F(a)|^2.
\]
Ancestors contribute only a smaller measure ratio in these suprema.
Choose $F$ supported in Fourier space on a minimizing or maximizing
character pair, as in \cref{prop:fourier-energy}. This proves optimality
and necessity of every weight comparison. If a weight is zero, this test
is a nonconstant input with zero energy, which rules out BMO recovery.
\end{proof}

\begin{example}[A scalar tight measurement]\label{ex:tight}
Take
\[
 v(0)=1,\qquad v(a)=-\frac1{q-1}\quad(a\ne0).
\]
Then $\norm v_\infty=1$ and
$\widehat v(\chi)=1/(q-1)$ for all $\chi\ne1$.
Consequently
\begin{equation}\label{eq:tight}
 \mathcal C_v(f)=\frac1{q-1}\norm f_{\BMO_2}.
\end{equation}
Indeed, \eqref{eq:translated-measurement} reads
$H_vF(b)=(F(b)-f_I)/(q-1)$. Equivalently,
\[
 T_{v,k}f(x)=\frac{\Delta_kf(x)}{q-1},
 \qquad
 \D_vf=\frac{1}{q-1}\D_1f,
\]
where $\D_1$ is the basic radial operator of Section~\ref{sec:radial}.
Thus the same residue kernel generates a radial maximal operator
and an energy functional recovering $\BMO_2$ exactly.
At a fixed center and scale this is a single scalar measurement.
Averaging its square over the centers detects every child direction,
and summing over scales gives \eqref{eq:tight}. In contrast, its maximal
operator is bounded by $\D_1^\res$, so
\cref{thm:resolution-obstruction} applies to it.
\end{example}

\begin{corollary}[Spectral moduli for energy]\label{cor:energy-moduli}
As $v$ ranges over mean-zero kernels \(v:\mathbb k\to\mathbb R^r\), allowing \(r\) to vary, its energy invariants
are exactly the functions
\[
 w:\widehat{\mathbb k}\setminus\{1\}\to[0,\infty),
 \qquad w(\chi)=w(\bar\chi).
\]
Every such invariant is realized by a real scalar kernel. For nonzero
invariants, normalization $\sum_{\chi\ne1}w(\chi)=1$ gives a simplex
with the conjugate coordinates constrained to agree. Unit-coordinate
changes permute the characters. Mutual comparison of energies is
equivalent to equality of the supports of their weights.
\end{corollary}
\begin{proof}
Set $\widehat v(1)=0$ and
$\widehat v(\chi)=\sqrt{w(\chi)}$ for $\chi\ne1$.
The inverse finite Fourier transform is real because the coefficients
on conjugate characters are conjugate, and it has mean zero.
Its weight function is $w$. All remaining assertions follow from
\cref{thm:energy}; on a finite character set, two nonnegative weight
functions with the same support dominate one another.
\end{proof}

The energy invariant forgets Fourier phases and records squared amplitudes.
The angular maximal invariant records a full norm on an angular quotient.
The general maximal invariant records a convex hull across all dilations.
These are three distinct, explicitly identified kinds of information.

\section{Scope and further questions}\label{sec:scope}
The classification depends on the kernel class, the input space,
and the equivalence relation. For angular kernels,
\cref{thm:angular-rigidity} gives a factorization criterion for
pointwise comparison on $L^1_{\loc}(X)$, with the same optimal
constant. Kernels constant on a fixed finite measurable partition
of the direction space admit finite-dimensional models described
by convex bodies, subspaces, and, for Hilbert-valued kernels,
Gram matrices (\cref{thm:finite-moduli,thm:gram}).
For radial residue kernels, \cref{thm:radial-classification}
gives a polynomial division criterion for comparison on
$L^\infty(K)$ and determines the optimal constants.
As \cref{ex:radial-domain} shows, this criterion does not extend
to all of $L^1_{\loc}(K)$.

At each fixed residue level and for a fixed normed output space,
residue kernels admit finite-dimensional parameter spaces and
exact one-scale models with fixed residue-class labels.
For the full maximal operators on $L^1_{\loc}(K)$,
\cref{thm:convex-invariant} gives a complete invariant:
the closed convex hull of the scalar averaging densities
obtained from all scales and the dual unit ball.
Equality of these sets characterizes equality of the operators,
while inclusion up to a constant factor characterizes
pointwise domination.

Several questions remain. The first is to determine when
comparison of two full residue maximal operators admits a
finite certificate expressed in terms of their residue tables,
extending \cref{prop:lp} beyond finite collections of scales.
The second is to describe the stabilizer strata of the
Gram-matrix quotients for specific local fields and determine
how these strata depend on the arithmetic unit-group action.
The third is to identify translation-invariant function spaces
and hypotheses on a fixed angular kernel under which its
maximal operator and sharp oscillation satisfy comparable
function-space norm estimates.
A further question is how the spectral classification in
\cref{thm:energy} extends to measurements involving several
consecutive scales, and whether a matrix-valued spectral
invariant describes the resulting Carleson functionals.

\appendix

\section{A guide to the equivalence relations}\label{app:guide}
For clarity, the classifications proved in the paper are collected below.
The distinction between a kernel, one average, and a full maximal operator
is part of every statement.

{\footnotesize
\begin{longtable}{@{}P{0.26\textwidth}P{0.28\textwidth}P{0.3\textwidth}@{}}

\toprule
Object & Equivalence & Complete invariant \\
\midrule
\endfirsthead

\toprule
Object & Equivalence & Complete invariant \\
\midrule
\endhead

General diamond maximal operator on \(L^1_{\loc}(X)\)
& Exact equality
& Closed convex hull of scalarized dilates, \cref{thm:convex-invariant}\\[1pt]
\midrule

Scalar angular operator
& Mutual pointwise domination
& Projective angular kernel, \cref{cor:scalar}\\[1pt]
\midrule

Angular operator on a fixed partition, normed output
& Exact equality
& Symmetric convex body \(C_v\), \cref{thm:finite-moduli}\\[1pt]
\midrule

Angular operator on a fixed partition, normed output
& Mutual pointwise domination
& Subspace \(\spanof C_v\), \cref{thm:finite-moduli}\\[1pt]
\midrule

Finite angular operator, Hilbert output
& Exact equality
& Gram matrix \(G_v\), \cref{thm:gram}\\[1pt]
\midrule

Residue table
& Scaling and unit coordinates
& \(\mathbb P(W_N)/U_N\), \cref{sec:residue}\\[1pt]
\midrule

Residue one-scale norm measurement
& Exact equality
& \(C_v\), or \(G_v\) for Hilbert outputs\\[1pt]
\midrule

Radial residue operator on bounded inputs
& Mutual pointwise domination
& Unit-circle zero divisor, \cref{thm:radial-classification}\\[1pt]
\midrule

Radial residue operator on bounded inputs
& Exact equality
& Filter up to sign and shift, \cref{thm:radial-classification}\\[1pt]
\midrule

Carleson functional \(\mathcal C_v\)
& Exact equality
& Fourier weights \(w_v\), \cref{thm:energy}\\

\bottomrule
\end{longtable}
}
For angular objects, geometric equivalence is obtained by quotienting these
invariants by the stated direction-group action. For residue tables, this
last operation does not remove the additional identifications introduced
by the supremum over scales.

\end{document}